\documentclass[reqno]{jgokova}
\usepackage{epsfig,graphicx}

\begin{document}

\newcommand{\ben}{\begin{enumerate}}
\newcommand{\een}{\end{enumerate}}
\newcommand{\be}{\begin{equation}}
\newcommand{\ee}{\end{equation}}
\newcommand{\bea}{\begin{eqnarray}}
\newcommand{\eea}{\end{eqnarray}}
\newcommand{\bc}{\begin{center}}
\newcommand{\ec}{\end{center}}

\newtheorem{thm}{Theorem}[section]
\newtheorem{cor}[thm]{Corollary}
\newtheorem{lem}[thm]{Lemma}
\newtheorem{prop}[thm]{Proposition}
\newtheorem{conj}[thm]{Conjecture}

\theoremstyle{definition}
\newtheorem{defn}[thm]{Definition}

\theoremstyle{remark}
\newtheorem{rem}[thm]{\rm\bfseries{Remark}}
\newtheorem*{notation}{Notation}

\newtheorem{ques}[thm]{\rm\bfseries{Question}}
\newtheorem{cons}[thm]{\rm\bfseries{Construction}}
\newtheorem{exm}[thm]{\rm\bfseries{Example}}


\newcommand{\C}{{\mathbb C}}
\newcommand{\R}{{\mathbb R}}
\newcommand{\Z}{{\mathbb Z}}
\newcommand{\Q}{{\mathbb Q}}
\renewcommand{\P}{{\mathbb P}}
\newcommand{\s}{{\mathbb S}}
\newcommand{\B}{{\mathbb B}}
\newcommand{\I}{{\mathbb I}}
\newcommand{\h}{{\mathbb H}}
\newcommand{\e}{{\mathbb E}}

              \newcommand{\J}{{\mathcal J}}
              \newcommand{\M}{{\mathcal M}}
              \newcommand{\W}{{\mathcal W}}
              \newcommand{\U}{{\mathcal U}}
              \newcommand{\T}{{\mathcal T}}
              \newcommand{\V}{{\mathcal V}}
              \newcommand{\E}{{\mathcal E}}
              \newcommand{\F}{{\mathcal F}}
              \renewcommand{\L}{{\mathcal L}}
              \renewcommand{\O}{{\mathcal O}}
              \newcommand{\N}{{\mathcal N}}
              \newcommand{\G}{{\mathcal G}}
              \renewcommand{\H}{{\mathcal H}}

\setcounter{page}{1}
\volume{3}

\title[Seiberg-Witten invariants of $3$-orbifolds]{Seiberg-Witten invariants of $3$-orbifolds and non-K\"{a}hler surfaces}
\author[CHEN]{Weimin Chen}

\thanks{The author is partially supported by NSF grant DMS-1065784.}

\address{Department of Mathematics and Statistics, University of Massachusetts-Amherst, Amherst, MA 01003, USA.}
\email{wchen@math.umass.edu}

\address{}
\email{}

\begin{abstract}
A formula is given which computes the Seiberg-Witten invariant of a $3$-orbifold from the invariant
of the underlying manifold. As an application, we derive a formula for the Seiberg-Witten invariant 
of a non-K\"{a}hler complex surface, which was originally due to O. Biquard \cite{Biq} and 
S.R. Williams \cite{W} independently. 
\end{abstract}
\keywords{Seiberg-Witten invariants, $3$-orbifolds, non-K\"{a}hler complex surfaces.}

\maketitle

\section{Introduction}
Let $Y$ be a closed, oriented $3$-orbifold with $b_1\geq 1$ whose singular set $\Sigma Y$ 
consists of a union of embedded circles. Consider a component $l\subset\Sigma Y$ with associated multiplicity $\alpha$, and let $Y_0$ be the oriented $3$-orbifold obtained from $Y$ by removing 
$l$ from the singular set. Our main result in this paper gives a relationship between the Seiberg-Witten invariants of $Y$ and $Y_0$. 

\begin{thm}
Let $S_Y,S_{Y_0}$ denote the set of $Spin^c$-structures on $Y,Y_0$ respectively. 
There exists a canonical map $\phi:S_Y\rightarrow S_{Y_0}$, which is surjective and $\alpha$ to $1$,
such that the Seiberg-Witten invariants of $Y$ and $Y_0$ obey the following equations 
$$
SW_Y(\xi)=SW_{Y_0}(\phi(\xi)), \;\; \forall \xi\in S_Y.
$$
\end{thm}

\begin{rem}
(1) The Seiberg-Witten invariant is defined as in Meng-Taubes \cite{MT}, Taubes \cite{T2}. See
Baldridge \cite{Bald2} for the extension to $3$-orbifolds. In particular, when $b_1=1$, the Seiberg-Witten invariants of $Y,Y_0$ depend on a choice of orientation of $H^1(|Y|;\R)$, where $|Y|$ stands
for  the underlying $3$-manifold. With this understood, we remark that the same orientation of 
$H^1(|Y|;\R)$ is used for $Y,Y_0$ in Theorem 1.1. Also, the sign of the Seiberg-Witten invariants
depends on a choice of homology orientation of $|Y|$; the same is used for $Y,Y_0$.

(2) A more explicit description of the map $\phi:S_Y\rightarrow S_{Y_0}$ is given in Section 2, 
together with a formula relating the determinant line bundles of the $Spin^c$-structures $\xi$,
$\phi(\xi)$, $\forall \xi\in S_Y$, cf. Proposition 2.6.

(3) Theorem 1.1 holds true more generally where $Y$ is allowed to have non-empty boundary
$\partial Y$, where $\partial Y$ consists of a union of tori (cf. \cite{MT, T2}), and $\Sigma Y\cap 
\partial Y=\emptyset$.

(4) Theorem 1.1 (together with Proposition 2.6) shows that, in contrast to geometric structures, the
Seiberg-Witten invariant of a $3$-orbifold $Y$ is much less sensitive to the embedding of the
singular set $\Sigma Y$ in $|Y|$, i.e. $\Sigma Y\subset |Y|$ as a knot or link.

(5) The Seiberg-Witten invariant of a $3$-manifold is equivalent to the Turaev torsion of the 
manifold \cite{MT, T2, Tu1,Tu2}. The equivalence of the two invariants may be extended to 
$3$-orbifolds by establishing an analog of Theorem 1.1 for the Turaev torsion.
\end{rem}

Theorem 1.1 has the following topological consequence. Recall that an orientable $3$-orbifold
is called {\it pseudo-good} if it does not contain any bad $2$-suborbifold. The importance of this
notion lies in the fact that the basic theory of $3$-manifolds was extended only to this class of 
$3$-orbifolds (see e.g. \cite{BMP}). It is known that Thurston's Geometrization Conjecture for
$3$-orbifolds (which is now proven, cf.  \cite{BLP}) implies that a pseudo-good $3$-orbifold must 
 be very good, i.e., it admits a finite, regular manifold cover (cf. \cite{MM}). 
 
Every pseudo-good $3$-orbifold admits a spherical splitting, i.e., the 
$3$-orbifold may be cut open along a system of spherical $2$-suborbifolds such that after capping off the boundary each component becomes irreducible. We remark that, unlike the connected sum 
decomposition of a $3$-manifold into prime factors, the spherical splitting of a $3$-orbifold is not 
known to be unique in general, and the issue of non-separating spherical $2$-suborbifolds has to be treated 
differently (cf. \cite{BMP}, page 41). However, when there are no non-separating spherical $2$-suborbifolds, Petronio showed in \cite{Pe} that the $3$-orbifold admits a so-called efficient 
spherical splitting for which the uniqueness statement holds. 

The following corollary of Theorem 1.1 follows easily from the fact that if a $3$-manifold with
$b_1>1$ contains a non-separating $2$-sphere, or is a connected sum of two manifolds both
of which have nonzero $b_1$, then its Seiberg-Witten invariant vanishes.

\begin{cor}
Let $Y$ be an oriented $3$-orbifold with $b_1>1$ and $SW_Y\neq 0$, whose singular set
$\Sigma Y$ consists of a union of embedded circles. 
Then $Y$ is pseudo-good
and every spherical $2$-suborbifold of $Y$ is separating. Moreover, in the efficient spherical 
splitting of $Y$ exactly one irreducible component has $b_1\neq 0$. 
\end{cor}

Corollary 1.3 has applications in the study of $4$-manifolds admitting a smooth, fixed-point free
circle action.

In this paper, we shall focus on another application of Theorem 1.1, where we derive a formula
for the Seiberg-Witten invariant of a non-K\"{a}hler complex surface. After this paper was 
completed, we found that a similar result was already obtained by O. Biquard \cite{Biq}
and S.R. Williams \cite{W} (independently)
using a different (and analytical) method. The Seiberg-Witten invariants of K\"{a}hler surfaces 
were determined by Brussee \cite{B} and Friedman-Morgan \cite{FM1} independently.

In order to state the formula, recall that according to the Enriques-Kodaira classification 
(cf. \cite{BPV}), a minimal complex surface $X$ with $b_2^{+}\geq 1$ is either a rational or ruled surface, or an elliptic surface, or a surface of general type. Moreover, if $X$ is non-K\"{a}hler,
it must be an elliptic surface with Euler number zero. According to \cite{FM}, Theorem 7.7, 
$X$ is obtained from the product $E\times C$, where $E=\C/\Lambda$ is an elliptic curve 
and $C$ is a curve of genus $g$, by doing logarithmic transforms on lifts $x_1,\cdots,x_n\in\C$ 
of $m_i$-torsion points $\xi_i$ modulo $\Lambda$, with $\sum_i x_i \neq 0$. Note that the assumption $b_2^{+}\geq 1$ implies that $g\neq 0$, which in turn implies that $b_2^{+}>1$.

Fix a basis ${\bf e}_1, {\bf e}_2$ of $\Lambda$. There are integers $a_i$, $b_i$ such that
$$
x_i=\frac{a_i{\bf e}_1+b_i{\bf e}_2}{m_i}, \;\; i=1,2,\cdots,n.
$$
Let $\Sigma$ be the $2$-orbifold whose underlying space is $C$, with singular points
$t_i$ of multiplicities $m_i$, where $t_i\in C$ is the point over which the logarithmic transform 
on $x_i$ is performed.
(Here we abuse the notation in the sense that when $m_i=1$, $t_i$ is in fact a regular point.)
With this understood, there are associated orbifold complex line bundles $E_1$, $E_2$ over
$\Sigma$ which
are defined as follows. Let $a_i=q_i m_i +u_i$, $b_i=r_i m_i+v_i$, where $0\leq u_i,v_i<m_i$.
Then the Seifert invariants of $E_1$, $E_2$ are 
$(\sum_i q_i, (m_i,u_i))$ and $(\sum_i r_i, (m_i,v_i))$ respectively. We remark that $E_1$, $E_2$
depend on the choice of the basis ${\bf e}_1$, ${\bf e}_2$, but the subgroup of orbifold complex
line bundles generated by $E_1$, $E_2$ depends only on $X$ (see Section 6 for more details).
We denote the subgroup by $\Gamma_X$. 

For any orbifold complex line bundle $D$ over $\Sigma$, we denote by $(D)$ the orbit of $D$ 
under the action of $\Gamma_X$, i.e., $(D)=(D^\prime)$ if and only if $D\equiv D^\prime \pmod{\Gamma_X}$. Finally, for an orbifold complex line bundle $D$ with Seifert invariant 
$(d, (m_i, s_i))$, we denote the ``background" degree $d$ by $|D|$. 

With the preceding understood, we have 

\begin{thm}
Let $X$ be a minimal, non-K\"{a}hler, complex surface with $b_2^{+}>1$. With the notations
introduced above, the set of $Spin^c$-structures on $X$ which have nonzero Seiberg-Witten invariant may be identified with a subset of orbits $(D)$, such that under this identification,
if a $Spin^c$-structure $\L$ corresponds to $(D)$ and $D$ has Seifert invariant $(d, (m_i,s_i))$, 
then in terms of Poincar\'{e} duality,
$$
c_1(\det \L)=(2d-2g+2)F +\sum_i (2s_i+1-m_i)F_i,
$$
where $F$ stands for a regular fiber and $F_i$ stands for the fiber at $t_i$ of the elliptic fibration
on $X$. Moreover, the Seiberg-Witten invariant of $X$ is given by
$$
SW_X(\L)\equiv SW_X((D))=\sum_{D^\prime\in(D)} (-1)^{|D^\prime|} \left (\begin{array}{c}
2g-2\\\
|D^\prime|\\
\end{array}
\right ).
$$
In the above formula, $ \left (\begin{array}{c}
2g-2\\
|D^\prime|\\
\end{array}
\right )=0$ if $|D^\prime|$ lies outside the interval $[0, 2g-2]$. 
\end{thm}

The organization of the remaining sections is as follows. In Section 2 we define the map 
$\phi: S_Y\rightarrow S_{Y_0}$ and discuss its relevant properties. Sections 3 and 4 are
devoted to the proof of Theorem 1.1, which is based on the gluing theorems of 
Morgan-Mrowka-Szab\'{o} \cite{MMS} and Taubes \cite{T2} respectively. In Section 5 we 
compute some examples to illustrate the theorem using Seifert $3$-manifolds. 
The proof of Theorem 1.4 about non-K\"{a}hler complex surfaces is given in 
Section 6 (see Theorem 6.2).

\vspace{3mm}

Part of the material in this paper has been circulated in a preprint ``Seifert fibered four-manifolds
with nonzero Seiberg-Witten invariant", arXiv:1103.5681v2 [math.GT].

\vspace{5mm}

\section{The map $\phi: S_Y\rightarrow S_{Y_0}$}

In order to define the map $\phi: S_Y\rightarrow S_{Y_0}$, we shall decompose $Y$ as a union
of $Y_{+}$ and $Y_{-}$ along a $2$-torus $T$, where $Y_{+}$ is the complement of a regular
neighborhood of $l$ in $|Y|$, and $Y_{-}$ is a regular neighborhood of $l$ in $Y$, which is 
a $3$-orbifold modeled by $D^2\times \s^1/\Z_\alpha$. In the above model, $D^2$ is the unit disc 
in $\C$ with the standard $\Z_\alpha$-action generated by a rotation of angle $2\pi/\alpha$, and the 
$\Z_\alpha$-action on $D^2\times \s^1$ is given by the product action which is trivial on the 
$\s^1$-factor. As $D^2$ is canonically oriented as a subset of $\C$, the orientation of $Y$ determines
an orientation of $l=\{0\}\times \s^1/\Z_\alpha$. Reversing the orientation of $l$ means reversing the
orientation of $D^2$. Likewise, we decompose $Y_0$ as a union of $Y_{+}$ and $Y_{0,-}$.
Here $Y_{0,-}$ is a regular neighborhood of $l$ in $Y_0$, identified with $D^2\times \s^1$,
where $D^2=|D^2/\Z_\alpha|$ is the underlying $2$-disc of the orbifold $D^2/\Z_\alpha$. 
The orientation of $l$ inherited from the identification $Y_{0,-}=D^2\times \s^1$ is the same as 
the one inherited from $Y_{-}=D^2\times\s^1/\Z_\alpha$.

Fixing such a decomposition for $Y$ (resp. $Y_0$), one can describe a $Spin^c$-structure
on $Y$ (resp. $Y_0$) in terms of $Spin^c$-structures on $Y_{+}$ and $Y_{-}$ (resp. $Y_{+}$
and $Y_{0,-}$) and some gluing datum along the $2$-torus $T=\partial Y_{+}=\partial Y_{-} (=
\partial Y_{0,-}$ respectively). To be more precise, let a normal direction along $T$ be chosen.
This is equivalent to fixing an orientation on $T$. Then any $Spin^c$-structure $\xi^{+}$ on 
$Y_{+}$ (likewise any $\xi^{-}$ on $Y_{-}$, or $\xi$ on $Y$) as an isomorphism class of principal 
$Spin^c(3)$-bundles lifting the $SO(3)$-bundle of oriented orthonormal frames admits a reduction 
to a principal $Spin^c(2)$-bundle over a neighborhood of $T$; that principal $Spin^c(2)$-bundle
corresponds to a $Spin^c$-structure on $T$, which we will denote by $\xi_{+}|_T$, called the
restriction of $\xi_{+}$ to $T$. It follows easily that $\xi_{+}|_T$ is isomorphic to the trivial 
$Spin^c$-structure on $T$, i.e., the unique $Spin^c$-structure on $T$ that has a trivial determinant line bundle, which admits a reduction to a $Spin$-structure on $T$. The $Spin$-structures on 
$T=\s^1\times \s^1$ are classified by $H^1(T;\Z_2)=\Z_2\times \Z_2$, and each $Spin$-structure 
gives a specific trivialization of the determinant line bundle of the $Spin^c$-structure on $T$. 
There is a special $Spin$-structure on $T$, denoted by $\xi_0$, which is the product of the non-trivial $Spin$-structure of each $\s^1$-factor of $T$. We remark that if $\xi_0$ is identified with $0\in
H^1(T;\Z_2)$, then for any other $Spin$-structure $\eta$ on $T$ which is identified with $\hat{\eta}
\in H^1(T;\Z_2)$, $\hat{\eta}$ is non-zero on a $\s^1$-factor of $T$ if and only if $\eta$ can be 
extended to a $Spin$-structure on a solid torus bounded by $T$ in which the $\s^1$-factor of $T$
bounds a disc. 

With the preceding understood, let $S_{Y_{+}}, S_{Y_{-}}$ be the set of $Spin^c$-structures on
$Y_{+}$, $Y_{-}$ respectively, and let  $S_{Y_{+}}^0, S_{Y_{-}}^0$ be the set of pairs 
$(\xi^{+},h_{+})$, $(\xi^{-},h_{-})$, where $\xi^{+}\in S_{Y_{+}}, \xi^{-}\in S_{Y_{-}}$,
and $h_{+}$, $h_{-}$ are homotopy classes of isomorphisms $\xi^{+}|_T\rightarrow \xi_0$,
$\xi^{-}|_T\rightarrow \xi_0$ respectively. Note that the automorphisms $\xi_0\rightarrow\xi_0$
are given by elements of $C^\infty(T;\s^1)$, whose homotopy classes may be identified with 
$H^1(T;\Z)$. Likewise, for any $\xi^{+}\in S_{Y_{+}}, \xi^{-}\in S_{Y_{-}}$, the automorphisms of 
$\xi^{+}$, $\xi^{-}$ are given by elements of $C^\infty(Y_{+};\s^1)$, $C^\infty(Y_{-};\s^1)$ respectively. The following lemma determines the set of homotopy classes of elements of $C^\infty(Y_{+};\s^1)$ and $C^\infty(Y_{-};\s^1)$.

\begin{lem}
Let $W$ be a compact orbifold (with or without boundary) whose singular set consists of a disjoint union of manifolds of co-dimension $2$. Let $|W|$ denote the underlying topological space and let 
$C^\infty(W; \s^1)$ denote the space of smooth circle-valued functions on the orbifold $W$. Then
there is a natural identification between $\pi_0(C^\infty(W;\s^1))$ and $H^1(|W|;\Z)$.
\end{lem}

\begin{proof}
Since the singular set of $W$ consists of a disjoint union of manifolds of co-dimension $2$, $|W|$ is
naturally a smooth manifold. Now given any $\varphi\in C^\infty(W;\s^1)$, it induces a continuous 
circle-valued function $\tilde{\varphi}$ on $|W|$. We denote by $[\varphi]$ the element in $H^1(|W|;\Z)$ 
determined by the homotopy class of $\tilde{\varphi}$. The correspondence $\varphi\mapsto [\varphi]$ clearly induces a mapping $\theta: \pi_0(C^\infty(W;\s^1))\rightarrow H^1(|W|;\Z)$. We will show that 
$\theta$ is both onto and one to one. 

Let $x\in H^1(|W|;\Z)$ be given. Since $|W|$ is a smooth manifold, there is a smooth circle-valued 
function $\tilde{\varphi}$ on $|W|$ which represents $x$. The pull back of $\tilde{\varphi}$ to $W$,
which we denote by $\varphi$, is smooth because locally the projection $W\rightarrow |W|$ is given 
by $(w,z)\mapsto (w,z^m)$ for some $m>1$, where $w\in \R^k$ and $z\in \C$. It is clear that 
$[\varphi]=x$, hence $\theta$ is onto. 

To see that $\theta$ is one to one, let $\varphi_1, \varphi_2\in C^\infty(W;\s^1)$ such that 
$[\varphi_1]=[\varphi_2]$. The induced continuous functions $\tilde{\varphi}_1$, $\tilde{\varphi}_2$
on $|W|$ are homotopic. We perturb $\tilde{\varphi}_1$, $\tilde{\varphi}_2$ into smooth functions 
$\tilde{\varphi}_{1+\epsilon}$, $\tilde{\varphi}_{2-\epsilon}$ on $|W|$ through a family of functions 
$\tilde{\varphi}_t$, with $1\leq t\leq 1+\epsilon$ and $2-\epsilon\leq t\leq 2$ respectively. The pull back of 
$\tilde{\varphi}_t$ to $W$ are continuous, but can be perturbed to a family of smooth functions on
the local uniformizing systems, which after averaging become equivariant with respect to the local group
actions on the uniformizing systems. In other words, $\varphi_1$, $\varphi_{1+\epsilon}$ and
$\varphi_2$, $\varphi_{2-\epsilon}$ are homotopic through smooth functions on $W$. Finally,
connect $\tilde{\varphi}_{1+\epsilon}$, $\tilde{\varphi}_{2-\epsilon}$ through a family of smooth 
functions $\tilde{\varphi}_t$ on $|W|$ with $1+\epsilon\leq t\leq 2-\epsilon$. Its pull back to $W$,
$\varphi_t$, are smooth functions on $W$. This shows that $\varphi_1, \varphi_2$ have the same
class in $\pi_0(C^\infty(W;\s^1))$, and therefore $\theta$ is one to one. 

\end{proof}

We consider natural actions of $H^1(T;\Z)$ on $S_{Y_{+}}^0$ and $S_{Y_{-}}^0$ defined as follows.
For any $e\in H^1(T;\Z)$, $(\xi^{+},h_{+})\in S_{Y_{+}}^0$, 
$$
e\cdot (\xi^{+},h_{+})=(\xi^{+},h_{+}+e),
$$
where $h_{+}+e$ denotes the homotopy class of isomorphisms $\xi^{+}|_{T}\rightarrow \xi_0$
which is given by $h_{+}$ followed by an automorphism of $\xi_0$ whose homotopy class 
is represented by $e$. Likewise, for any $e\in H^1(T;\Z)$, $(\xi^{-},h_{-})\in S_{Y_{-}}^0$, 
$$
e\cdot (\xi^{-},h_{-})=(\xi^{-},h_{-}+e). 
$$
With this understood, $H^1(|Y_{+}|;\Z)$ and $H^1(|Y_{-}|;\Z)$ act on $S_{Y_{+}}^0$ and 
$S_{Y_{-}}^0$ via homomorphisms $j_{+}: H^1(|Y_{+}|;\Z)\rightarrow H^1(T;\Z)$ and 
$j_{-}: H^1(|Y_{-}|;\Z)\rightarrow H^1(T;\Z)$, which are induced by the embedding of $T$ in 
$|Y_{+}|, |Y_{-}|$. As a consequence of
Lemma 2.1, we note that if $e\in j_{+}H^1(|Y_{+}|;\Z)$ (reps. $e\in j_{-}H^1(|Y_{-}|;\Z)$), then
for any $(\xi^{+},h_{+})\in S_{Y_{+}}^0$ (reps. $(\xi^{-},h_{-})\in S_{Y_{-}}^0$), there is an automorphism $\tilde{e}$ of $\xi^{+}$ (reps. $\xi^{-}$) such that $h_{+}\circ \tilde{e}|_{\xi^{+}|_{T}}=h_{+}+e$ (resp. $h_{-}\circ \tilde{e}|_{\xi^{-}|_{T}}=h_{-}+e$). We denote by $\underline{S}_{Y_{+}}$,
$\underline{S}_{Y_{-}}$ the quotient set of $S_{Y_{+}}^0$ and $S_{Y_{-}}^0$ under the action of
$H^1(|Y_{+}|;\Z)$ and $H^1(|Y_{-}|;\Z)$ respectively.  Finally, define an action of $H^1(T;\Z)$
on $\underline{S}_{Y_{+}}\times \underline{S}_{Y_{-}}$ by
$$
e\cdot ([(\xi^{+},h_{+})], [(\xi^{-},h_{-})])=([(\xi^{+},h_{+}+e)], [(\xi^{-},h_{-}+e)]), 
$$
where $e\in H^1(T;\Z), [(\xi^{+},h_{+})]\in \underline{S}_{Y_{+}},  [(\xi^{-},h_{-})]\in 
\underline{S}_{Y_{-}}$.

The same discussions apply to $Y_0$, and we obtain the set $\underline{S}_{Y_{0,-}}$ and 
a similarly defined action of $H^1(T;\Z)$ on $\underline{S}_{Y_{+}}\times \underline{S}_{Y_{0,-}}$.

\begin{lem}
There are natural identifications $\psi: \underline{S}_{Y_{+}}\times \underline{S}_{Y_{-}}/
H^1(T;\Z)\rightarrow S_Y$ and $\psi_0: \underline{S}_{Y_{+}}\times 
\underline{S}_{Y_{0,-}}/H^1(T;\Z)\rightarrow S_{Y_0}$.
\end{lem}

\begin{proof}
The map $\psi$ is defined as follows. For any $((\xi^{+},h_{+}), (\xi^{-},h_{-}))\in 
{S}^0_{Y_{+}}\times {S}^0_{Y_{-}}$, we construct a $Spin^c$-structure on $Y$
by gluing $\xi^{+}$ and $\xi^{-}$ along $T$ via the isomorphism $h_{-}^{-1}\circ h_{+}: 
\xi^{+}|_{T}\rightarrow \xi^{-}|_{T}$, and assign this $Spin^c$-structure to 
$((\xi^{+},h_{+}), (\xi^{-},h_{-}))$. Call this map $\tilde{\psi}: {S}^0_{Y_{+}}\times {S}^0_{Y_{-}}
\rightarrow S_Y$. Clearly $\tilde{\psi}$ factors through 
$\underline{S}_{Y_{+}}\times \underline{S}_{Y_{-}}/H^1(T;\Z)$.

First, we check that $\tilde{\psi}$ is surjective onto $S_Y$. For any $\xi\in S_Y$, set $\xi^\pm 
\equiv \xi|_{Y_\pm}\in S_{Y_\pm}$.  Then $\xi^\pm |_T=\xi|_T$. We pick any $h_{+}:\xi^{+}|_T
\rightarrow \xi_0$, and let $h_{-}=h_{+}$ under the identification $\xi^{-}|_T=\xi^{+}|_T$. Then
it follows easily that $((\xi^{+},h_{+}), (\xi^{-},h_{-}))$ is sent to $\xi$ under $\tilde{\psi}$.

Next we check that $\tilde{\psi}$ is one to one after factoring through the quotient set. To this end,
let $((\xi^{+}_i,h_{+,i}), (\xi^{-}_i,h_{-,i}))$, $i=1,2$, be sent to the same image under
$\tilde{\psi}$. Then there are isomorphisms $f_\pm: \xi^{\pm}_1\rightarrow \xi^{\pm}_2$, such 
that 
$$
h_{-,2}^{-1}\circ h_{+,2}\circ f_{+}|_T=f_{-}|_T \circ h_{-,1}^{-1}\circ h_{+,1}.
$$
Without loss of generality, we may assume $\xi^{\pm}_1=\xi^{\pm}_2$, and consequently,
$f_{\pm}$ are given by elements of $C^\infty(Y_{\pm};\s^1)$. Then the class of $f_{\pm}|T$
in $H^1(T;\Z)$, $[ f_{\pm}|T]$, lies in the image $j_{\pm}: H^1(|Y_{\pm}|;\Z)\rightarrow H^1(T;\Z)$. Replacing $h_{+,2}$ by $ h_{+,2}\circ [f_{+}|_T]$, and $h_{-,1}$ by $[f_{-}|_T] \circ h_{-,1}$, which
give the same elements in $\underline{S}_{Y_{+}}$ and $\underline{S}_{Y_{-}}$ respectively,
we may assume instead 
$$
h_{-,2}^{-1}\circ h_{+,2}=h_{-,1}^{-1}\circ h_{+,1}.
$$
Let $e\in H^1(T;\Z)$ such that $h_{+,2}=h_{+,1}+e$, then $h_{-,2}=h_{-,1}+e$ also holds. This
finishes the proof that $\psi$ is one to one and onto. 

The definition of the map $\psi_0$ and the verification that $\psi_0$ is one to one and onto are completely analogous.

(We remark that Lemma 2.2 is analogous to Lemma 2.6 in Taubes \cite{T2}. However, the formulation
of Lemma 2.6 in \cite{T2} is ambiguous when there are $2$-torsions in the second cohomology.
We thank Cliff Taubes for clarifying this issue for us.)

\end{proof}

With Lemma 2.2, the promised map $\phi:S_Y\rightarrow S_{Y_0}$ will be defined through a
$H^1(T;\Z)$-equivariant map $\phi^\prime: \underline{S}_{Y_{-}}\rightarrow \underline{S}_{Y_{0,-}}$.
In order to define $\phi^\prime$, we first take a closer look at the set of $Spin^c$-structures on $Y_{-}$
and $Y_{0,-}$ respectively. 

As for $Y_{-}=D^2\times \s^1/\Z_\alpha$, an element of $S_{Y_{-}}$ may be regarded as a 
$\Z_\alpha$-equivariant $Spin^c$-structure on $D^2\times \s^1$. This said, there are $\alpha$ elements of $S_{Y_{-}}$, labelled by $\xi^{-}_\beta$, where $\beta=0,1,\cdots,\alpha-1$. It is more convenient to describe $\xi^{-}_\beta$ in terms of the associated $\Z_\alpha$-equivariant spinor bundle $S^{-}_\beta$ on $D^2\times \s^1$. With this understood, for $\beta=0$, $S^{-}_0$ is the 
$\Z_\alpha$-equivariant bundle on $D^2\times \s^1$ given by $\I\oplus K^{-1}_{D^2}$, where $\I$ 
is the trivial $\Z_\alpha$-equivariant complex line bundle and $K_{D^2}$ is the canonical bundle of $D^2$ which is regarded as a $\Z_\alpha$-equivariant complex line bundle on $D^2\times \s^1$. Moreover, for each $\beta\neq 0$, $S^{-}_\beta$ is the $\Z_\alpha$-equivariant bundle 
$S^{-}_0\otimes \underline{\C}_\beta$, where 
$\underline{\C}_\beta$ is the $\Z_\alpha$-equivariant complex line bundle on $D^2\times \s^1$ 
given by $D^2\times \s^1\times \C$ with the $\Z_\alpha$-action 
$$
\lambda\cdot (z, x, w)= (\lambda z,x, \lambda^\beta w), \;\;\;\lambda=\exp(2\pi i/\alpha), 
(z,x)\in D^2\times \s^1, w\in\C. 
$$

The set $S_{Y_{0,-}}$, $Y_{0,-}=D^2\times \s^1$, consists of only one element, denoted by 
$\xi^{0,-}$. 

\begin{lem}
For $(\xi^{-},h_{-})\in S_{Y_{-}}^0$, let $c(\xi^{-},h_{-})$ denote the first Chern class of the determinant
line bundle $\det \xi^{-}$ of $\xi^{-}$ equipped with the trivialization given by $\det h_{-}:
\det \xi^{-}|_T\rightarrow \det\xi_0$. (Here $c(\xi^{-},h_{-})$ is defined via Chern-Weil theory and lives
in the relative cohomology group $H^2(|Y_{-}|,T;\R)$.) Likewise, let $c(\xi^{0,-},h_{0,-})$ be the
first Chern class for $(\xi^{0,-},h_{0,-})\in S_{Y_{0,-}}^0$. Then
\begin{itemize}
\item $\int_{D^2\times \{pt\}} c(\xi^{0,-},h_{0,-})$ takes values in the set $\{2k+1|k\in\Z\}$.
\item For any $(\xi^{-},h_{-})\in S_{Y_{-}}^0$, $\int_{D^2\times \{pt\}} c(\xi^{-},h_{-})$ takes values 
in the set 
$$
\{2k+\frac{2\beta+1}{\alpha} |k\in\Z, \beta=0,1,\cdots, \alpha-1\}.
$$
\end{itemize}
\end{lem}

\begin{proof}
Consider $\int_{D^2\times \{pt\}} c(\xi^{0,-},h_{0,-})$, $\forall (\xi^{0,-},h_{0,-})\in S_{Y_{0,-}}^0$, first.
Observe that $\xi^{0,-}$ admits a reduction to a $Spin$-structure on $Y_{0,-}=D^2\times \s^1$.
We fix any such a $Spin$-structure $\eta$, and denote by $h: \xi^{0,-}|_T\rightarrow \eta|_T$ the
corresponding isomorphism on $T$. Then $\eta$ gives rise to a trivialization of $\det \xi^{0,-}$ on $Y_{0,-}$, with the trivialization of $\det \xi^{-,0}|_T$ given by 
$\det h: \det\xi^{0,-}|_T\rightarrow \det\eta|_T$. With this understood, $\int_{D^2\times \{pt\}} c(\xi^{0,-},h)=0$,
and 
$$
\int_{D^2\times \{pt\}} c(\xi^{0,-},h_{0,-})=\int_{D^2\times \{pt\}} c(\xi^{0,-},h)+
\det h_{0,-}\circ \det h^{-1} ([\partial D^2]),
$$
where $\det h_{0,-}\circ \det h^{-1}$ is regarded as an element  of $H^1(T;\Z)$ and $\partial D^2$
is oriented as boundary of $D^2$. 
Since $\eta|_T$ extends to $\eta$ over $D^2\times \s^1$, the difference between 
the $Spin$-structures $\eta_T$ and
$\xi_0$ is given by a class $\hat{\eta}\in H^1(T;\Z_2)$ which evaluates nontrivially on $\partial D^2$.
Consequently, $\det h_{0,-}\circ \det h^{-1} ([\partial D^2])=\hat{\eta}([\partial D^2])=1 \pmod{2}$,
and $\int_{D^2\times \{pt\}} c(\xi^{0,-},h_{0,-})=2k+1$ for some $k\in\Z$.

Next we compute $\int_{D^2\times \{pt\}} c(\xi^{-},h_{-})$,  $\forall (\xi^{-},h_{-})\in S_{Y_{-}}^0$.
To this end, we introduce an oriented $3$-orbifold $\Sigma\times \s^1$, 
where $\Sigma$ is the $2$-orbifold whose underlying $2$-manifold 
$|\Sigma|=\s^2$ and $\Sigma$ has exactly one singular point of multiplicity 
$\alpha$. It is easily seen that $\Sigma\times \s^1$ admits a decomposition as a union of 
$Y_{+}$ and $Y_{-}$ along the common boundary $T$ with $Y_{+}=D^2\times \s^1$.

There is a canonical $Spin^c$-structure $\xi^0$ on $\Sigma\times \s^1$ whose associated spinor bundle $S^0$ is given by $\I\oplus K^{-1}_{\Sigma}$. With this canonical $Spin^c$-structure $\xi^0$,
any $Spin^c$-structure $\xi$ on $\Sigma\times\s^1$ corresponds to an orbifold complex line bundle $E$ (always a pull-back from $\Sigma$) such that the associated spinor bundle $S$ is given
by $S^0\otimes E$. Let $(b, (\alpha,\beta))$, where $b\in \Z$ and $0\leq \beta<\alpha$, be the
Seifert invariant of $E$. Then
$$
c_1(\det \xi)(|\Sigma|)= 2b+\frac{2\beta}{\alpha}+ 2+(\frac{1}{\alpha}-1)
=2b+1+\frac{2\beta+1}{\alpha}.
$$

Now we apply the identification $\psi$ in Lemma 2.2 to $Y=\Sigma\times \s^1$, and via $\psi$ we identify $\xi$ to a pair of elements $(\xi^{+},h_{+})$, $(\xi^{-},h_{-})$ (note that
$\xi^{-}=\xi^{-}_\beta$ by construction). Then
$$
2b+1+\frac{2\beta+1}{\alpha}=c_1(\det \xi)(|\Sigma|)=\int_{D^2\times \{pt\}} c(\xi^{+},h_{+})+\int_{D^2\times \{pt\}} c(\xi^{-},h_{-}).
$$
By the previous calculation, $\int_{D^2\times \{pt\}} c(\xi^{+},h_{+})=2s+1$, $s\in\Z$. Thus
$$
\int_{D^2\times \{pt\}} c(\xi^{-},h_{-})\equiv \int_{D^2\times \{pt\}} c(\xi^{-}_\beta,h_{-})
=2k+\frac{2\beta+1}{\alpha}, \; k\in\Z.
$$
This completes the proof of the lemma.

\end{proof}

Note that the correspondences 
$$
(\xi^{0,-},h_{0,-})\mapsto \int_{D^2\times \{pt\}} c(\xi^{0,-},h_{0,-}),\;\;
(\xi^{-},h_{-})\mapsto \int_{D^2\times \{pt\}} c(\xi^{-},h_{-})
$$
factor through the quotient sets $\underline{S}_{Y_{0,-}}$ and $\underline{S}_{Y_{-}}$. This gives
definitions of the values
$$
c([(\xi^{0,-},h_{0,-})])\equiv \int_{D^2\times \{pt\}} c(\xi^{0,-},h_{0,-}),\;\;
c([(\xi^{-},h_{-})])\equiv \int_{D^2\times \{pt\}} c(\xi^{-},h_{-}),
$$
which allow us to distinguish elements of $\underline{S}_{Y_{0,-}}$ and $\underline{S}_{Y_{-}}$,
because the group of orbifold complex line bundles on $Y_{-}$ with a fixed trivialization on the boundary is torsion-free. 

With the preceding understood, the definitions of the maps $\phi^\prime$ and $\phi$ are given
as follows.

\begin{defn}
(1) Recall that the elements of $S_{Y_{-}}$ are labelled by $\xi^{-}_\beta$, 
$0\leq \beta< \alpha$. There are two characteristics of $\xi^{-}_\beta$,
any one of which uniquely determines $\xi^{-}_\beta$: (i) the spinor bundle of $\xi^{-}_\beta$
is given by the $\Z_\alpha$-equivariant bundle $(\I\oplus K^{-1}_{D^2})\otimes \underline{\C}_\beta$
on $D^2\times \s^1$, and (ii) (cf. the proof of Lemma 2.3)
$$
c([(\xi^{-}_\beta, h_{-})])=2k+\frac{2\beta+1}{\alpha}, \;\; k\in\Z.
$$
With this understood, we define 
$\phi^\prime:\underline{S}_{Y_{-}}\rightarrow \underline{S}_{Y_{0,-}}$ which is uniquely
determined by the following condition: for any $[(\xi^{-}_\beta, h_{-})]\in \underline{S}_{Y_{-}}$,
$\phi^\prime([(\xi^{-}_\beta, h_{-})])\in  \underline{S}_{Y_{0,-}}$ such that if 
$c([(\xi^{-}_\beta, h_{-})])=2k+\frac{2\beta+1}{\alpha}$ for some $k\in \Z$ uniquely picked out
by $h_{-}$, we require $\phi^\prime([(\xi^{-}_\beta, h_{-})])$ to satisfy 
$$
c(\phi^\prime([(\xi^{-}_\beta, h_{-})]))=2k +1.
$$
The map $\phi^\prime$ is clearly $H^1(T;\Z)$-equivariant.

(2) With the identifications $\psi, \psi_0$ in Lemma 2.2, the map $\phi:S_Y\rightarrow S_{Y_0}$ 
is defined to be the one induced by the $H^1(T;\Z)$-equivariant map from $\underline{S}_{Y_{+}}\times \underline{S}_{Y_{-}}$ to $\underline{S}_{Y_{+}}\times \underline{S}_{Y_{0,-}}$ given by
$$
([(\xi_{+},h_{+})], [(\xi_{-},h_{-})])\mapsto ([(\xi_{+},h_{+})], \phi^\prime ([(\xi_{-},h_{-})])).
$$
\end{defn}

\begin{rem}
We remark that although the definition of $\phi$ involves fixing a number of auxiliary data, such
as an orientation of $l$ and an orientation of $T$, as the identifications 
$Y_{-}=D^2\times \s^1/\Z_\alpha$ and $Y_{0,-}=D^2\times \s^1$, the reduction of a 
$Spin^c$-structure in a neighborhood of $T$ to a $Spin^c$-structure on $T$, as well
as the definitions of $c([(\xi^{0,-},h_{0,-})])$ and $c([(\xi^{-},h_{-})])$ all require it, 
one can easily verify that reversing the orientation of $l$ or $T$ gives rise to the same map $\phi$. 
Hence $\phi$ is intrinsically defined. However, take note that the naming of $\xi^{-}_\beta$ 
depends on the choice of orientation of $l$.
\end{rem}

For any $\xi\in S_Y$, the determinant line bundle of $\xi$ can be nicely related to the determinant
line bundle of $\phi(\xi)$. Such a relation between $\det \xi$ and $\det \phi(\xi)$ is useful in 
computations. 

In order to state the said relation between $\det \xi$ and $\det \phi(\xi)$, we need to introduce some
notations. For each $k\in \Z$, we denote by $E_k$ the orbifold complex line bundle on $Y$ defined
as follows. One first defines $E_k$ on $Y_{-}=D^2\times \s^1/\Z_\alpha$ by the $\Z_\alpha$-equivariant complex line bundle $\underline{\C}_k$ on $D^2\times \s^1$, i.e., the one given by $D^2\times \s^1\times\C$ with the $\Z_\alpha$-action
$$
\lambda\cdot (z, x, w)= (\lambda z,x, \lambda^k w), \;\;\;\lambda=\exp(2\pi i/\alpha), 
(z,x)\in D^2\times \s^1, w\in\C. 
$$
This orbifold complex line bundle has a canonical trivialization on $\partial Y_{-}=\partial D^2\times \s^1$ given by the $\Z_\alpha$-equivariant nonzero section $(z,x)\mapsto z^k$. The bundle $E_k$
is defined over the rest of $Y$ by extending this trivialization. Observe two useful facts about $E_k$:
(1) $E_k$ descends to a complex line bundle on $|Y|$ if and only if $k$ is divisible by
$\alpha$, in which case the descendant of $E_k$ to $|Y|$ has $c_1$ given by $\frac{k}{\alpha}\cdot
PD(l)$ in $H^2(|Y|;\Z)$, and (2) when $l$ is non-torsion, the first Chern class of $E_k$ as an 
element of $H^2(|Y|;\R)$ is given by $\frac{k}{\alpha}\cdot PD(l)$ (cf. \cite{C0}, Lemma 3.6).

With the preceding understood, we have 

\begin{prop}
Given any $\xi\in S_Y$, suppose $\xi|_{Y_{-}}$ is given by $\xi^{-}_\beta$ for some 
$\beta=\beta_\xi$, $0\leq \beta<\alpha$. Then 
$$
\det \phi(\xi)=\det \xi \otimes E_{\alpha-2\beta_\xi-1}.
$$
\end{prop}

\begin{proof}
We use $\psi$ in Lemma 2.2 to identify $\xi$ with the orbit of $([(\xi^{+}, h_{+})],
[(\xi^{-}_\beta, h_{-})])$. Then $\det \xi$ is given by $([(\det \xi^{+}, \det h_{+})],
[(\det \xi^{-}_\beta, \det h_{-})])$, where $\det h_{+}: \det \xi^{+}|_T \rightarrow \det \xi_0$ and
$\det h_{-}: \det \xi^{-}_\beta|_T \rightarrow \det \xi_0$. Moreover, we normalize our choice by
requiring that 
$$
c([(\xi^{-}_\beta, h_{-})])=\frac{2\beta+1}{\alpha}.
$$

With the preceding understood, recall that  the spinor bundle of $\xi^{-}_\beta$ is given by
$(\I\oplus K^{-1}_{D^2})\otimes \underline{\C}_\beta$. Thus $\det \xi^{-}_\beta=K^{-1}_{D^2}\otimes
\underline{\C}_{2\beta}$. Noticing that $K^{-1}_{D^2}=\underline{\C}_1$, so that 
$\det \xi^{-}_\beta =\underline{\C}_{2\beta+1}$,  we find that the trivialization of $\det \xi^{-}_\beta$ determined by the $\Z_\alpha$-equivariant non-zero section $(z,x)\mapsto z^{2\beta+1}$ gives
the same evaluation of the relative first Chern class of $\det \xi^{-}_\beta$ on $D^2\times \{pt\}$
as that determined by $\det h_{-}$  (cf. Lemma 3.6 of \cite{C0}). Hence the trivialization of 
$\det h_{-}$ may be identified with the trivialization given by the non-zero section 
$(z,x)\mapsto z^{2\beta+1}$.

With this understood, note that $\det \xi \otimes E_{\alpha-2\beta-1}$ is given by the pair 
$$
([(\det \xi^{+}, \det h_{+})], [(\det \xi^{-}_\beta\otimes E_{\alpha-2\beta-1}|_{Y_{-}}, h)]),
$$
where, with $\det \xi^{-}_\beta\otimes E_{\alpha-2\beta-1}|_{Y_{-}}=\underline{\C}_\alpha$, 
$h: (\det \xi^{-}_\beta\otimes E_{\alpha-2\beta-1}|_{Y_{-}})|_T \rightarrow \det \xi_0$ is given by
the $\Z_\alpha$-equivariant non-zero section $(z,x)\mapsto z^\alpha$, $(z,x)\in \partial D^2\times
\s^1=T$.  It follows that 
$$
\int_{D^2\times \{pt\}} c_1(\det \xi^{-}_\beta\otimes E_{\alpha-2\beta-1}|_{Y_{-}},h)=1=
c(\phi^\prime ([(\xi^{-}_\beta, h_{-})])),
$$
which implies that $\det \phi(\xi)=\det \xi \otimes E_{\alpha-2\beta_\xi-1}$ as claimed. 
Hence the proposition.

\end{proof}

Now we shall proceed to prove Theorem 1.1. Our strategy is to consider $4$-orbifolds 
$M=\s^1\times Y$ and $M_0=\s^1\times Y_0$. By a theorem of Donaldson \cite{D}, the
Seiberg-Witten invariant of $M$ for a $Spin^c$-structure $\L$ is non-zero only if $\L$ is the 
pull-back of a $Spin^c$-structure $\xi$ on $Y$, and in this case one has
$$
SW_M(\L)=SW_Y(\xi).
$$
The same thing is true for $M_0$ and $Y_0$. Thus, we shall consider the subset $S_M$ of 
$Spin^c$-structures on $M$ which are pull-backs of a $Spin^c$-structure on $Y$, and likewise,
consider the subset $S_{M_0}$ of $Spin^c$-structures on $M_0$ that are pull-backs of a 
$Spin^c$-structure on $Y_0$.  Note that $S_M,S_{M_0}$ may be canonically identified with $S_Y,S_{Y_0}$, so that there is correspondingly a surjective, $\alpha:1$ map, also 
denoted by $\phi$, from $S_M$ to $S_{M_0}$. Theorem 1.1 follows if one shows that 
$$
SW_M (\L)=SW_{M_0}(\phi(\L)), \;\; \forall \L\in S_M.
$$

By setting $M_{+}=\s^1\times Y_{+}$, $M_{-}=\s^1\times Y_{-}$, and $M_{0,-}=\s^1\times Y_{0,-}$,
and $N=\s^1\times T$ which is a $3$-torus,  one similarly obtains decompositions of $M$ as a union
of $M_{+}$ with $M_{-}$ along $N$, and $M_0$ as a union of $M_{+}$ with $M_{0,-}$ along $N$.

There are similarly defined sets $\underline{S}_{M_{+}}$, $\underline{S}_{M_{-}}$ and 
$\underline{S}_{M_{0,-}}$, equipped with group actions by $H^1(N;\Z)$, which are canonically
identified with the sets $\underline{S}_{Y_{+}}$, $\underline{S}_{Y_{-}}$ and 
$\underline{S}_{Y_{0,-}}$. The actions by $H^1(N;\Z)$ and $H^1(T;\Z)$ are compatible with
the understanding that $H^1(T;\Z)$ embeds in $H^1(N;\Z)$ as a subgroup which is induced by the projection $N=\s^1\times T\rightarrow T$. Moreover, there are identifications, still denoted by 
$\psi,\psi_0$,
$$
\psi: \underline{S}_{M_{+}}\times \underline{S}_{M_{-}}/H^1(N;\Z)\rightarrow S_M \mbox{ and }
\psi_0: \underline{S}_{M_{+}}\times \underline{S}_{M_{0,-}}/H^1(N;\Z)\rightarrow S_{M_0},
$$
and a $H^1(N;\Z)$-equivariant map $\phi^\prime: \underline{S}_{M_{-}}\rightarrow 
\underline{S}_{M_{0,-}}$, such that $\phi: S_M\rightarrow S_{M_0}$ is given by 
$Id\times \phi^\prime/ H^1(N;\Z)$.  We continue to use the notations introduced in the 
$3$-dimensional setting for the corresponding objects in the $4$-dimensional setting.

With the preceding understood, the equations 
$$
SW_M (\L)=SW_{M_0}(\phi(\L)), \;\; \forall \L\in S_M,
$$
will be derived from the gluing theorems of Seiberg-Witten invariants along $T^3$, as developed in
Morgan, Mrowka and Szab\'{o} \cite{MMS}, as well as in Taubes \cite{T2}. This said, we actually 
need the extensions of the gluing theorems to $4$-orbifolds. Since the singular sets of the $4$-orbifolds are all lying in the complement of the gluing region, there are no essential complications in the analysis involved. 

We remark that each of the gluing theorems has its own limitations. For instance, Taubes' gluing
theorem requires that the $3$-torus $N$ be essential (in the sense of \cite{T2}), and 
Morgan-Mrowka-Szab\'{o}'s gluing theorem needs the assumption that $b^{+}_2(M_{+})>0$.
Thus the proof of Theorem 1.1 is given separately in two cases.
\begin{itemize}
\item [{(i)}] $b^{+}_2(M_{+})>0$, which means that $H_1(|Y_{+}|,T;\Z)\neq 0$. In the case of
$b_1(Y)=1$, this condition is equivalent to $l$ is torsion in $H_1(|Y|;\Z)$.
\item [{(ii)}] $N$ is essential, which is equivalent to the condition that $l$ is non-torsion in
$H_1(|Y|;\Z)$.
\end{itemize}

Note that Case (i) and Case (ii) have overlaps, which means that for these cases, the gluing 
theorems of Morgan-Mrowka-Szab\'{o} and Taubes give independent proofs for Theorem 1.1.

\section{Proof of Theorem 1.1: case (i)}

Let $\L\in S_M$ be given. Via $\psi$ in Lemma 2.2 we identify $\L$ with a pair of elements 
$(\xi^{+}, h_{+})$, $(\xi^{-}_\beta, h_{-})$, where $\xi^{+}=\L|_{M_{+}}$, $\xi^{-}_\beta=
\L|_{M_{-}}$ for some $0\leq \beta<\alpha$, and $h_{+}: \xi^{+}|_N\rightarrow \xi_0$,
$h_{-}: \xi^{-}_\beta|_N\rightarrow \xi_0$ are isomorphisms. Recall that this identification means
that $\L$ is obtained by gluing $\xi^{+}$ with $\xi^{-}_\beta$ along $N$ via $h_{-}^{-1}\circ h_{+}$.
We shall fix $h_{+}$ throughout, and in doing so, $h_{-}$ is uniquely determined up to the
actions by $H^1(|M_{+}|;\Z)$ and $H^1(|M_{-}|;\Z)$. Let $\Lambda=\{h_{-,i}\}$ be a set of elements
of such $h_{-}$'s each of which represents an orbit of the $H^1(|M_{-}|;\Z)$-action. Furthermore,
we shall normalize our choice of $h_{+}$ by requiring that there is a $h_{-,0}\in\Lambda$
such that $c([(\xi^{-}_\beta,h_{-,0})])=(2\beta+1)/\alpha$.

The set $\Lambda$ has the following properties. Let $e_1$ be the vector in $H^1(N;\Z)$ 
which is the Poincar\'{e} dual of $\{pt\}\times T^2$ in $N$. Then if $l$ is torsion of order $m$, the image $j_{+} H^1(|M_{+}|;\Z)\subset
H^1(N;\Z)$ is generated by $me_1$. If $l$ is non-torsion, then no non-zero multiples of $e_1$ is 
contained in $j_{+} H^1(|M_{+}|;\Z)$. With this understood, in the former case we have
$c([(\xi^{-}_\beta,h_{-})])=(2\beta+1)/\alpha \pmod{2m}, \forall h_{-}\in\Lambda$, and in the
latter case, $\Lambda$ consists of a single element $h_{-}$ with 
$c([(\xi^{-}_\beta,h_{-})])=(2\beta+1)/\alpha$. 

Analogously, via $\psi_0$ in Lemma 2.2 we identify $\phi(\L)$ with a pair of elements 
$(\xi^{+}, h_{+})$, $(\xi^{0,-}, h_{0,-})$, where $(\xi^{+}, h_{+})$ is the one chosen in the 
previous paragraph, and $h_{0,-}: \xi^{0,-}|_N\rightarrow \xi_0$ is an isomorphism chosen
out of a set $\Lambda_0$, which is given by the image of the set $\Lambda$ under the
map $\phi^\prime$. When $l$ is torsion of order $m$, $c([(\xi^{0,-},h_{0,-})])
=1 \pmod{2m}, \;\forall h_{0,-}\in\Lambda_0$, and if $l$ is non-torsion, $\Lambda_0$ consists of
a single element $h_{0,-}$ with $c([(\xi^{0,-},h_{0,-})])=1$. 

Finally, the Dirac operator $D$ associated to the $Spin$-structure $\xi_0$ has $\ker D=\C^2$,
which consists of covariantly constant sections. This gives rise to a trivialization of $\xi_0$.
Viewing differently, if we regard $\xi_0$ as a $Spin^c$-structure, then it follows that there is a 
special
flat $U(1)$-connection $\theta_0$ on $\det \xi_0$, such that the associated Dirac operator
$D_{\theta_0}$ has $\ker D_{\theta_0}=\C^2$. This gives rise to an identification of the space of
equivalence classes of flat $U(1)$-connections on $\det \xi_0$ modulo homotopically trivial
gauge transformations with the space $H^1(N;\R)$ via harmonic $1$-forms, under which 
$\theta_0$ is identified with the origin $0\in H^1(N;\R)$. See \cite{MMS} for more details.

With the preceding understood, we now review the gluing theorem of Seiberg-Witten invariant of Morgan, Mrowka and Szab\'{o} in \cite{MMS} (adapted to the present orbifold context).  The key component of this analysis is the structure of $L^2$-moduli spaces of Seiberg-Witten equations 
on $4$-manifolds (or more generally $4$-orbifolds) with cylindrical ends $T^3 \times [0,\infty)$. We consider $M_{+}$ first, which is an oriented $4$-orbifold with boundary $N=T^3$. We endow 
$M_{+}$ with a Riemannian metric $g$ that is flat near $N$ and attach $N \times [0,\infty)$ to it extending the metric $g$ naturally. Call the resulting cylindrical-end Riemannian $4$-orbifold 
$(\hat{M}_{+}, \hat{g})$. 

With $\L\in S_M$ given, $\xi^{+}\in S_{M_{+}}$ being the restriction of $\L$ to $M_{+}$, we extend
$\xi^{+}$ naturally to a $Spin^c$-structure on $\hat{M}_{+}$. With this understood, one considers 
the $L^2$-moduli space of Seiberg-Witten equations 
$$
\left\{ \begin{array}{c}
{D_A}\psi=0 \\
P_{+}F_A=\tau(\psi\otimes\psi^\ast)-i\mu
\end{array}
\right .
$$
where $\mu$ is a compactly supported real-valued smooth self-dual $2$-form. The $L^2$-moduli 
space, denoted by $\M_{\hat{M}_{+}}(\xi^{+}, \hat{g}, \mu)$,  is the space of $(A,\psi)$ modulo
gauge transformations by the group $C^\infty (\hat{M}_{+};\s^1)$, where $(A,\psi)$ satisfies the
finite energy condition 
$$
\int_{\hat{M}_{+}} |F_A|^2 dvol<\infty.
$$

Let $\chi_0(N)$ be the space of flat $U(1)$-connections on $\det \xi^{+}|_N$ modulo gauge 
transformations by those $\varphi\in C^{\infty}(N;\s^1)$ where $\varphi$ is the restriction of 
an element of $C^\infty (\hat{M}_{+};\s^1)$. There is a continuous map 
$\partial_\infty: \M_{\hat{M}_{+}}(\xi^{+}, \hat{g}, \mu)\rightarrow \chi_0(N)$ which sends the gauge equivalence class of $(A,\psi)$ to the class of the asymptotic value of $A|_{N\times \{t\}}$ as
$t\rightarrow \infty$. With the isomorphism $h_{+}:\xi^{+}|_N \rightarrow \xi_0$ chosen, $\chi_0(N)$
is identified with $H^1(N;\R)/j_{+}2H^1(|M_{+}|;\Z)$ through $\det h_{+}: \det \xi^{+}|_N 
\rightarrow \det \xi_0$, as the corresponding space of flat $U(1)$-connections on $\det \xi_0$ is identified with $H^1(N;\R)/j_{+} 2H^1(|M_{+}|;\Z)$ via harmonic $1$-forms.
As $h_{+}$ is being fixed throughout, we shall regard 
$\partial_\infty$ as a map from $M_{\hat{M}_{+}}(\xi^{+}, \hat{g}, \mu)$ into 
$H^1(N;\R)/j_{+}2H^1(|M_{+}|;\Z)$. 

With the preceding understood, the structure theorem of Morgan, Mrowka and Szab\'{o} (cf. \cite{MMS}, Theorem 2.8) asserts that for a generic choice of $\mu$, the $L^2$-moduli space 
$\M_{\hat{M}_{+}}(\xi^{+}, \hat{g}, \mu)$ is a compact $1$-dimensional manifold with boundary, 
such that the image of $\partial \M_{\hat{M}_{+}}(\xi^{+}, \hat{g}, \mu)$ under the map 
$\partial_\infty$ lies in the lattice of points of even integer coordinates in 
$H^1(N;\R)/j_{+} 2H^1(|M_{+}|;\Z)$. We remark that this assertion requires the condition that
$b_2^{+}(M_{+})>0$, which is assumed to be true. Furthermore, we used the fact that the 
dimension of the moduli space $\M_\L$ of Seiberg-Witten equations associated to the 
$Spin^c$-structure $\L$ is $0$. (This follows easily from the fact that $\L$ is the pull-back of a 
$Spin^c$-structure on $Y$.) 
Indeed, according to \cite{MMS}, the dimension of $\M_{\hat{M}_{+}}(\xi^{+}, \hat{g}, \mu)$ 
is given by $\dim \M_\L +1$ (cf. Theorem 2.4 in \cite{MMS}), which equals $1$ since 
$\dim \M_\L=0$.

Similarly, we consider the $L^2$-moduli space on $\hat{M}_{-}$, 
$\M_{\hat{M}_{-}}(\xi^{-}_\beta,\hat{g},\mu)$, where $\xi^{-}_\beta\in S_{M_{-}}$ is the 
restriction of $\L$ on $M_{-}$. Since $M_{-}=D^2\times T^2/\Z_\alpha$, we can choose a metric 
$\hat{g}$ which has non-negative, somewhere positive scalar curvature. Taking $\mu=0$, it follows
that $\M_{\hat{M}_{-}}(\xi^{-}_\beta,\hat{g},\mu)$ consists entirely of reducible solutions, i.e., 
$(A,0)$ where $A$ is flat. For the purpose of gluing we shall in fact consider the based version
$\M^0_{\hat{M}_{-}}(\xi^{-}_\beta,\hat{g},\mu)$, which is the space of finite energy solutions modulo
homotopically trivial gauge transformations.  Thus $\M^0_{\hat{M}_{-}}(\xi^{-}_\beta,\hat{g},\mu)$
is given by the set of $[A]$'s, where $[A]$ is the equivalence class of flat $U(1)$-connections on 
$\det \xi^{-}_\beta$ modulo homotopically trivial gauge transformations. Now for each 
$h_{-}\in \Lambda$,  the isomorphism $\det h_{-}:\det \xi^{-}_\beta|_N \rightarrow \det \xi_0$ 
embeds $M^0_{\hat{M}_{-}}(\xi^{-}_\beta,\hat{g},\mu)$ into $H^1(N;\R)$ as follows: $[A]\in M^0_{\hat{M}_{-}}(\xi^{-}_\beta,\hat{g},\mu)$ is sent to the image of $[A]|_N$ in $H^1(N;\R)$ via the map induced by $\det h_{-}$. For each $h_{-}\in\Lambda$, we denote by 
$\Omega_\beta(h_{-})$ the corresponding image in $H^1(N;\R)/ j_{+}2H^1(|M_{+}|;\Z)$. 

The $L^2$-moduli space on $\hat{M}_{0,-}$ has a similar structure. More precisely, we consider 
$\M_{\hat{M}_{0,-}}(\xi^{0,-},\hat{g},\mu)$, the $L^2$-moduli space on $\hat{M}_{0,-}$.  Again, 
since $M_{0,-}=D^2\times T^2$, we can choose a metric $\hat{g}$ which has non-negative, somewhere positive scalar curvature. Taking $\mu=0$, it follows that 
$\M_{\hat{M}_{0,-}}(\xi^{0,-},\hat{g},\mu)$ consists entirely of reducible solutions. 
Let $\M^0_{\hat{M}_{0,-}}(\xi^{0,-},\hat{g},\mu)$ be the based version. Then for each $h_{0,-}\in \Lambda_0$, the isomorphism $\det h_{0,-}$ embeds $\M^0_{\hat{M}_{0,-}}(\xi^{0,-},\hat{g},\mu)$ 
into $H^1(N;\R)$. Denote by $\Omega_{0,-}(h_{0,-})$ the corresponding image in 
$H^1(N;\R)/ j_{+} 2H^1(|M_{+}|;\Z)$, $\forall h_{0,-}\in\Lambda_0$.

\begin{lem}
Regard $e_1$ as a coordinate function on 
$H^1(N;\Z)/ j_{+} 2H^1(|M_{+}|;\Z)$, which takes values only in $\R/2m\Z$ when $l$
is torsion of order $m$. Then for any $h_{-}\in\Lambda$, $h_{0,-}\in\Lambda_0$, the subsets 
$\Omega_\beta (h_{-})$ and $\Omega_{0,-}(h_{0,-})$ are $2$-dimensional subspaces defined
by $e_1=-(2\beta+1)/\alpha$ and $e_1=-1$ respectively. Consequently, $\Omega_\beta (h_{-})
\equiv\Omega_\beta$, $\Omega_{0,-}(h_{0,-})\equiv \Omega_{0,-}$ are independent of 
$h_{-}$, $h_{0,-}$ respectively, and both $\Omega_\beta$ and $\Omega_{0,-}$ miss the lattice of
even integer coordinates in $H^1(N;\R)/ j_{+}2H^1(|M_{+}|;\Z)$. Moreover, $\Omega_\beta$ is 
isotopic to $\Omega_{0,-}$ in the complement of the lattice of even integer coordinates in 
$H^1(N;\R)/ j_{+}2H^1(|M_{+}|;\Z)$.
\end{lem}

\begin{proof}
First, consider the embedding of $\M^0_{\hat{M}_{0,-}}(\xi^{0,-},\hat{g},\mu)$ into $H^1(N;\R)$
via $\det h_{0,-}$ for any given $h_{0,-}\in \Lambda_0$. Let $[A]\in
\M^0_{\hat{M}_{0,-}}(\xi^{0,-},\hat{g},\mu)$ be any element. Note that the holonomy of $[A]|_N$
is zero around the loop $\partial D^2$. By modifying $A$ with an imaginary valued harmonic
$1$-form which vanishes on $\partial D^2$, we can assume that $[A]|_N$ has integral holonomy,
thus defining a trivialization $h$ of $\det \xi^{0,-}|_N$. Note that this modification does not change
the $e_1$-coordinate of the image of $[A]$ in $H^1(N;\R)$ since the harmonic $1$-form used
vanishes on $\partial D^2$. With respect to the trivialization $h$, the relative first Chern class of 
$\det \xi^{0,-}$ is zero. It follows then that
$$
c([(\xi^{0,-},h_{0,-})])=-h\circ h_{0,-}^{-1} ([\partial D^2])=-e_1([A]),
$$
where $e_1([A])$ stands for the $e_1$-coordinate of the image of $[A]$ in $H^1(N;\R)$.
This shows that the $e_1$-coordinate on $\Omega_{0,-}(h_{0,-})$, taking values in 
$\R/2m\Z$ when $l$ is torsion of order $m$, is constant equaling $-1$. Consequently, 
$\Omega_{0,-}(h_{0,-})\equiv \Omega_{0,-}$ is independent of $h_{0,-}\in\Lambda_0$.
It is the $2$-dimensional subspace defined by $e_1=-1$ because 
$\M^0_{\hat{M}_{0,-}}(\xi^{0,-},\hat{g},\mu)$ is a $2$-dimensional space. Clearly,
$\Omega_{0,-}$ misses the lattice of even integer coordinates in 
$H^1(N;\R)/ j_{+}2H^1(|M_{+}|;\Z)$ since $e_1=-1$ on $\Omega_{0,-}$.

The case of $\Omega_\beta(h_{-})$ is similar. For any $[A]\in 
\M_{\hat{M}_{-}}(\xi^{-}_\beta,\hat{g},\mu)$, we regard it as a flat $U(1)$-connection on 
$D^2\times T^2$, where $M_{-}=D^2\times T^2/\Z_\alpha$. Let $\hat{N}=\partial
(D^2\times T^2)$ which covers $N$, and let $\hat{e}_1$ be the corresponding coordinate
on $H^1(\hat{N};\R)$. Then $\hat{e}_1=\alpha \cdot e_1$. With this understood, one obtains
as argued in the previous paragraph that
$$
\alpha \cdot c([(\xi^{-}_\beta,h_{-})])=-\hat{e}_1([A]).
$$
Hence $e_1([A])=-c([(\xi^{-}_\beta,h_{-})])=-(2\beta+1)/\alpha$. Analogously, 
$\Omega_\beta(h_{-})\equiv \Omega_\beta$ is independent of $h_{-}\in \Lambda$, and is
the $2$-dimensional subspace defined by $e_1=-(2\beta+1)/\alpha$. Since for
$0\leq\beta<\alpha$, $e_1=-(2\beta+1)/\alpha$ takes no even integer values, $\Omega_\beta$
also misses the lattice of even integer coordinates in $H^1(N;\R)/ j_{+}2H^1(|M_{+}|;\Z)$.

Finally, each $\Omega_\beta$ is isotopic to $\Omega_{0,-}$ in the complement of the lattice
of even integer coordinates because $e_1=-(2\beta+1)/\alpha$ lies in the interval
$[-2+\frac{1}{\alpha},-\frac{1}{\alpha}]$, which contains no even integers and contains $-1$,
the value of the $e_1$-coordinate of $\Omega_{0,-}$. 

\end{proof}

With the preceding understood, the boundary map 
$$
\partial_\infty: \M_{\hat{M}_{+}}(\xi^{+}, \hat{g}, \mu)\rightarrow H^1(N;\R)/ j_{+}2H^1(|M_{+}|;\Z)
$$ 
can be made transverse to each subspace $\Omega_\beta$, $\Omega_{0,-}$ by taking a generic choice of $\mu$. Then the gluing theorem of Morgan, Mrowka and Szab\'{o} asserts that 
$$
SW_M(\L)=\# (\M_{\hat{M}_{+}}(\xi^{+}, \hat{g}, \mu)\cap 
\partial_\infty^{-1}(\Omega_\beta))
$$
and 
$$
SW_{M_0}(\phi(\L))=\# (\M_{\hat{M}_{+}}(\xi^{+}, \hat{g}, \mu)\cap 
\partial_\infty^{-1}(\Omega_{0,-})),
$$
where the right-hand sides stand for the algebraic intersections of 
$\M_{\hat{M}_{+}}(\xi^{+}, \hat{g}, \mu)$ with $\partial_\infty^{-1}(\Omega_\beta)$ and 
$\partial_\infty^{-1}(\Omega_{0,-})$ respectively.
The assertion $SW_M(\L)=SW_{M_0}(\phi(\L))$ follows from the fact that $\Omega_\beta$
is isotopic to $\Omega_{0,-}$ in the complement of the lattice of even integer coordinates in 
$H^1(N;\R)/ j_{+}2H^1(|M_{+}|;\Z)$.

\section{Proof of Theorem 1.1: case (ii)}

The $3$-torus $N$ being essential in $M$ (resp. $M_0$) means that there is a class 
$\varpi\in H^2(|M|;\R)$ whose restriction to $H^2(N;\R)$ is non-trivial, which is easily seen
to be equivalent to $l$ being non-torsion in $H_1(|Y|;\Z)$. The existence of such a $\varpi$ 
is what is required in the setup of Taubes' gluing theorem in \cite{T2}. This said, we shall fix a covariantly constant $2$-form $\omega_0$ on $N$ (with respect to a fixed flat metric) which represents the restriction of $\varpi$ in $H^2(N;\R)$.

With the preceding understood, let $\L\in S_M$ be given. We shall endow $M$ with a Riemannian 
metric $g$ such that in a regular neighborhood of $N$, $g=ds^2+g_N$, where $s$ is the normal
coordinate in the regular neighborhood and $g_N$ is a fixed flat metric on $N$.
The Seiberg-Witten equations considered in this context take the following form
$$
\left\{ \begin{array}{c}
{D_A}\psi=0 \\
P_{+}F_A=\tau(\psi\otimes\psi^\ast)-i r\omega
\end{array}
\right .
$$
where $\omega$ is a self-dual $2$-form and $r>0$ is a fixed, sufficiently large constant. 
Moreover, in a regular neighborhood of $N$, $\omega=ds\wedge \theta+\omega_0$ 
where $\theta$ is the metric dual to $\omega_0$. 

The idea of gluing is to analyze the effect on the moduli space of Seiberg-Witten equations when 
the cylindrical neck neighborhood of $N$ gets metrically longer and longer. This requires knowledge 
about the corresponding $L^2$-moduli spaces over the cylindrical-end orbifolds $\hat{M}_{+}$,
$\hat{M}_{-}$ respectively. 

Consider $\hat{M}_{+}$ first. Given any $\xi^{+}\in S_{M_{+}}$, we let $\M(\xi^{+})$ denote the space of solutions $(A,\psi)$ of the above form of the Seiberg-Witten equations modulo gauge transformations, 
where $(A,\psi)$ satisfies the following finite energy condition
$$
\int_{\hat{M}_{+}} |F_A|^2 dvol<\infty.
$$
(Here on the cylindrical end of $\hat{M}_{+}$, one assumes that $\omega=ds\wedge \theta+\omega_0$.)

According to \cite{T2}, Lemma 3.5, for each gauge equivalence class $[(A,\psi)]\in\M(\xi^{+})$, 
the restriction of $(A,\psi)$ to the slice $\{s\}\times N$ on the cylindrical end converges, up to a gauge
transformation, exponentially fast to a $(A_0,\psi_0)$ as $s\rightarrow\infty$, where $A_0$ is a trivial connection on $\det (\xi^{+}|_N)$ and $\psi_0$ is a non-zero, covariantly constant section of the spinor bundle on $N$. This said, each $(A,\psi)$ determines a trivialization of $\xi^{+}|_N$, which is given by a homotopy class of isomorphisms from $\xi^{+}|_N$ to $\xi_0$. Consequently, each gauge 
equivalence class  $[(A,\psi)]$ determines an element $[(\xi^{+}, h_{+})]\in\underline{S}_{M_{+}}$,
and accordingly, there is a decomposition 
$$
\M(\xi^{+})=\bigcup_{[(\xi^{+}, h_{+})]\in \underline{S}_{M_{+}}} \M([(\xi^{+},h_{+})]).
$$
Furthermore, for any constant $C$, there are only finitely many $[(\xi^{+},h_{+})]$ 
such that $\M([(\xi^{+},h_{+})])\neq \emptyset$ and 
$$
\int_{|M_{+}|} c(\xi^{+},h_{+})\cup \varpi \leq C.
$$
Finally, each $\M([(\xi^{+},h_{+})])$ can be used to define a Seiberg-Witten invariant, 
which we will denote by $SW([(\xi^{+},h_{+})])\in\Z$.

Similar discussions apply to $\hat{M}_{-}$, $\hat{M}_{0,-}$ as well.  More precisely, given any 
$\xi_\beta^{-}\in S_{M_{-}}$, we have the corresponding $L^2$-moduli space $\M(\xi_\beta^{-})$ 
and a decomposition
$$
\M(\xi^{-}_\beta)=\bigcup _{[(\xi^{-}_\beta,h_{-})]\in \underline{S}_{M_{-}}} 
\M([(\xi_\beta^{-},h_{-})]),
$$
and for $\xi^{0,-}\in S_{M_{0,-}}$, we have 
$$
\M(\xi^{0,-})=\bigcup _{[(\xi^{0,-},h_{0,-})]\in \underline{S}_{M_{0,-}}} \M([(\xi^{0,-},h_{0,-})]).
$$
Finally, Seiberg-Witten invariants $SW([(\xi^{-}_\beta,h_{-})])\in\Z$, 
$SW([(\xi^{0,-},h_{0,-})])\in\Z$ are defined using the moduli spaces 
$\M([(\xi_\beta^{-},h_{-})])$ and $\M([(\xi^{0,-},h_{0,-})])$ respectively.

With the preceding understood, Taubes' gluing theorem (cf. \cite{T2}, Theorem 2.7) asserts that
$$
SW_M(\L)=\sum_{([(\xi^{+},h_{+})], [(\xi^{-}_\beta,h_{-})])\in \psi^{-1}(\L)}
SW([(\xi^{+},h^{+})])\cdot  SW([(\xi^{-}_\beta,h^{-})])
$$
and
$$
SW_M(\phi(\L))=\sum_{([(\xi^{+},h_{+})], [(\xi^{0,-},h_{0,-})])\in \psi_0^{-1}(\phi(\L))}
SW([(\xi^{+},h^{+})])\cdot  SW([(\xi^{0,-},h^{0,-})]),
$$
where $\psi$, $\psi_0$ are the maps in Lemma 2.2.  

By the definition of the map $\phi$, it is clear that $([(\xi^{+},h_{+})], [(\xi^{0,-},h_{0,-})])\in \psi_0^{-1}(\phi(\L))$ if and only if $[(\xi^{0,-},h_{0,-})]=\phi^\prime ([(\xi^{-}_\beta,h_{-})])$ and
$([(\xi^{+},h_{+})], [(\xi^{-}_\beta,h_{-})])\in \psi^{-1}(\L)$. Thus the assertion
$SW_M(\L)=SW_{M_0}(\phi(\L))$ follows immediately from the following lemma.

\begin{lem}
For any $[(\xi^{-}_\beta,h_{-})]\in\underline{S}_{M_{-}}$,
$
SW([(\xi^{-}_\beta,h_{-})])=SW(\phi^\prime([(\xi^{-}_\beta,h_{-})])).
$
\end{lem}

\begin{proof}
According to Taubes \cite{T2}, the Seiberg-Witten invariant of $M_{0,-}=D^2\times T^2$ is given
by the polynomial 
$$
t(1-t^2)^{-1}=t+t^3+t^5+\cdots.
$$
This can be equivalently stated as 
$$
SW([(\xi^{0,-},h_{0,-})])=\left\{\begin{array}{cc}
1 & \mbox{ if } c([(\xi^{0,-},h_{0,-})])>0\\
0 & \mbox{ otherwise. }\\
\end{array}
\right .
$$
Thus the lemma follows by showing that for any $[(\xi^{-}_\beta,h_{-})]\in\underline{S}_{M_{-}}$,
$$
SW([(\xi^{-}_\beta,h_{-})])=\left\{\begin{array}{cc}
1 & \mbox{ if } c([(\xi^{-}_\beta,h_{-})])>0\\
0 & \mbox{ otherwise, }\\
\end{array}
\right .
$$
because by the definition of the map $\phi^\prime$, 
$$
c([(\xi^{-}_\beta,h_{-})])>0 \mbox{ if and only if } c(\phi^\prime([(\xi^{-}_\beta,h_{-})]))>0.
$$

The strategy of the proof is based on the following observation: by Taubes' gluing theorem in \cite{T2},
Lemma 4.1 must be true as long as the claimed relation $SW_M(\L)=SW_{M_0}(\phi(\L))$ holds 
for any one particular example of $4$-orbifold $M$ which has nonzero Seiberg-Witten invariant.
For example, consider $M=T^2\times \Sigma$ where $\Sigma$ is the $2$-torus with one
singular point of multiplicity $\alpha$. The corresponding $M_0$ is simply the $4$-torus
which has nonzero Seiberg-Witten invariant.
For this $4$-orbifold, the relation $SW_M(\L)=SW_{M_0}(\phi(\L))$ has been established in
the previous section using the gluing theorem of Morgan, Mrowka and Szab\'{o}, which would then
imply Lemma 4.1. 

For a somewhat independent proof of the lemma, we can take $M=T^2\times \Sigma$ where 
$\Sigma$ is the $2$-sphere with one singular point of multiplicity $\alpha$. In this case,
$SW_M(\L)=SW_{M_0}(\phi(\L))$ can be established using the wall-crossing formula in Li-Liu \cite{LL}. 

More concretely, since $M$ has positive scalar curvature, the Seiberg-Witten invariant of $M$
is given by the wall-crossing number, which equals (cf. Lemma 2.6 in \cite{LL})
$$
\int_{T^{b_1}} c_{\frac{b_1}{2}}(V^{+}-V_{-}).
$$
Here $T^{b_1}=H^1(|M|;\R)/H^1(|M|;\Z)$ is the torus parametrizing the gauge equivalence classes
of reducible solutions, and $V_{+}-V_{-}$ is the index bundle of the family of Dirac operators 
parametrized by $T^{b_1}$. The key observation is that in this case $M$ is K\"{a}hler, so that the 
Dirac operators are given by the $d$-bar operators
(cf. \cite{M}). The kernel and co-kernel of the $d$-bar operators, which are given by holomorphic sections of orbifold bundles over an orbifold with 
co-dimension $2$ singularities, can be identified with the kernel and co-kernel of the 
corresponding $d$-bar operators for the de-singularization of the orbifold bundles over the
de-singularization of the orbifold. 

With this understood, let $\L$ be a $Spin^c$-structure on $M$ whose spinor bundle is given
by $(\I\oplus K^{-1}_{\Sigma})\otimes E$, where $E$ is an orbifold complex line bundle over 
$\Sigma$
with Seifert invariant $(b, (\alpha,\beta))$, $0\leq \beta<\alpha$. In this case, the de-singularization of $E$ is the complex line bundle $E_0$ over $\s^2$ (the de-singularization of $\Sigma$) which has
degree $b$. Thus the Seiberg-Witten invariant $SW_M(\L)$, given by the wall-crossing number,
is equal to the Seiberg-Witten invariant of $M_0$ associated to the $Spin^c$-structure whose spinor bundle is $(\I\oplus K^{-1}_{\s^2})\otimes E_0$ (which is given by the same wall-crossing
number). Note that this $Spin^c$-structure on $M_0$ is exactly $\phi(\L)$. Finally, the Seiberg-Witten invariant of $M_0=T^2\times \s^2$ is nonzero, from which Lemma 4.1 follows. 

\end{proof}

\section{Examples: Seifert $3$-manifolds}

Let $Y=\s^1\times \Sigma$ where $\Sigma$ is a Riemann surface of genus $g$. Then the
Seiberg-Witten invariant of $Y$ is given by the following polynomial 
$$
SW_Y(t)=(t^{-1}-t)^{2g-2}, \mbox{ where } t=PD(\s^1\times \{pt\}).
$$
(In $SW_Y(t)$, the coefficient of $t^k$ is the Seiberg-Witten invariant $SW_Y(\xi)$ where the 
$Spin^c$-structure $\xi$ satisfies $c_1(\det \xi)=k t$.) This formula can be obtained independently
in two different ways. The first method is to identify $SW_Y$ with the Seiberg-Witten invariant 
of the K\"{a}hler surface $T^2\times\Sigma$ and then appeal to the formula in Brussee \cite{B} or
Friedman-Morgan \cite{FM1}. The second method is to use Taubes' gluing theorem in \cite{T2}
to run an induction on the genus $g$. The latter approach requires initial values for the cases
where $g=0$ or $1$, i.e., the Seiberg-Witten invariant of $T^2\times \s^2$ or $T^4$. The Seiberg-Witten invariant of $T^2\times \s^2$ can be computed using the wall-crossing formula
in \cite{LL}, and the Seiberg-Witten invariant of $T^4$ follows from Taubes' work \cite{T1}.

For our purpose here, we shall reinterpret the formula for $SW_Y(t)$ as follows. Note that $Y$
has a canonical $Spin^c$-structure $\xi_0$, whose associated spinor bundle $S_0$ is given by
$\I\oplus K_\Sigma^{-1}$. Any relevant $Spin^c$-structure $\xi$ (i.e., with nonzero Seiberg-Witten invariant) may be identified with a complex line bundle $D$ over $\Sigma$ such that the associated
spinor bundle of $\xi$ is given by $D\oplus  K_\Sigma^{-1} \otimes D$. With this understood,
note that $c_1(\det\xi)=(2d-2g+2)t$ where $d$ is the degree of $D$, and we can reinterpret the
formula for $SW_Y(t)$ as 
$$
SW_Y(\xi)\equiv SW_Y(D)=\left \{\begin{array}{ll}
(-1)^d \left (\begin{array}{c}
2g-2\\
d
\end{array}
\right ) & \mbox{ if } g\geq 1, d\in [0,2g-2],\\
d+1 & \mbox{if } g=0, \; d\geq 0,\\
0 & \mbox{ otherwise. }\\
\end{array}
\right .
$$

As a straightforward application of Theorem 1.1, we shall compute the Seiberg-Witten invariant
of the $3$-orbifold $Y=\s^1\times \Sigma$, where $\Sigma$ is the $2$-orbifold whose underlying
surface $|\Sigma|$ has genus $g$ and whose singular set consists of $m$ points $z_1,\cdots,
z_m$ with the associated multiplicities $\alpha_1,\cdots,\alpha_m$. To this end, note that $Y$
has a canonical $Spin^c$-structure $\xi_0$ whose associated spinor bundle is $S_0=\I\oplus
K_{\Sigma}^{-1}$. Any $Spin^c$-structure $\xi$ with nonzero Seiberg-Witten invariant can be
identified with an orbifold complex line bundle $D$ over $\Sigma$ in the sense that the associated spinor bundle of $\xi$ is $S=D\oplus  K_\Sigma^{-1} \otimes D$. Now observe that 
if a $Spin^c$-structure $\xi$ on $Y$ is given by $D$ over $\Sigma$, then the $Spin^c$-structure 
$\phi(\xi)$ on $|Y|=\s^1\times |\Sigma|$ is given by $|D|$ over $|\Sigma|$, where, if $D$ is given 
by the Seifert invariant $(d, (\alpha_1,\beta_1), \cdots, (\alpha_m,\beta_m))$, 
with $0\leq \beta_i<\alpha_i$, the complex line bundle $|D|$ over $|\Sigma|$ has degree $d$. 
With the preceding understood, Theorem 1.1 implies that
$$
SW_Y(\xi)\equiv SW_Y(D)=SW_{|Y|}(|D|)=\left \{\begin{array}{ll}
(-1)^d \left (\begin{array}{c}
2g-2\\
d
\end{array}
\right ) & \mbox{ if } g\geq 1, d\in [0,2g-2],\\
d+1 & \mbox{if } g=0, \; d\geq 0,\\
0 & \mbox{ otherwise. }
\end{array}
\right .
$$

Having been able to compute the Seiberg-Witten invariant of the $3$-orbifolds 
$\s^1\times\Sigma$, we shall next determine the Seiberg-Witten invariant of Seifert $3$-manifolds
$Y$, where $\pi: Y\rightarrow \Sigma$ is the unit sphere bundle of an orbifold complex line bundle
$E$ over $\Sigma$. We assume the Seifert invariant of $E$ is 
$$
(e, (\alpha_1,e_1),\cdots, (\alpha_m,e_m)), \mbox{ where } 0<e_i<\alpha_i, \text{ gcd }(e_i,\alpha_i)=1.
$$
We remark that if $b_1(Y)\geq 1$, then the Euler class of $E$ must be torsion when the genus
of $|\Sigma|$ is zero, i.e., $g=0$.

The Seiberg-Witten invariant of $Y$ can be identified with the Seiberg-Witten invariant of the
$4$-manifold $\s^1\times Y$, which Seifert fibers over the $3$-orbifold $\s^1\times \Sigma$. 
Work of Baldridge \cite{Bald2} then allows us to compute the Seiberg-Witten invariant of
$\s^1\times Y$ in terms of that of $\s^1\times\Sigma$. More precisely, according to \cite{Bald2},
the relevant $Spin^c$-structures on $\s^1\times Y$ are pull-backs of $Spin^c$-structures on
$\s^1\times \Sigma$. With this said, we shall let $\xi_0$ denote the $Spin^c$-structure on
$\s^1\times Y$ which is the pull-back of the canonical $Spin^c$-structure on $\s^1\times\Sigma$
whose associated spinor bundle is given by $\I\oplus K^{-1}_\Sigma$. Then any relevant 
$Spin^c$-structure $\xi$ may be identified with a complex line bundle $\pi^\ast D$, where $D$ is an orbifold complex line bundle over $\Sigma$. Note that the set of complex line bundles $\pi^\ast D$ may be identified with the set of equivalence classes $[D]$, where $[D]=[D^\prime]$ if and only if
$D-D^\prime \equiv 0 \pmod{E}$. 

With the preceding understood, work of Baldridge (cf. \cite{Bald2}, Theorem C) implies that 
the Seiberg-Witten invariant of $Y$ is given by the following formula
$$
SW_Y(\xi)\equiv SW_Y([D])=\sum_{D^\prime\in [D]} SW_{\s^1\times \Sigma} (D^\prime).
$$
(Compare the work of Mrowka, Ozsv\'{a}th and Yu \cite{MOY}; the special case of the above
formula where $\Sigma$ is non-singular has been discussed in \cite{Bald1}.)
It is useful to observe that if a $Spin^c$-structure $\xi$ corresponds to $[D]$, then $c_1(\det \xi)=
\pi^\ast (c_1(2D-K_\Sigma))$. In terms of Poincar\'{e} duality, if $D$ is given by the Seifert invariant
$(d, (\alpha_1,\beta_1),\cdots, (\alpha_m,\beta_m))$ and if $F$, $F_i$ denote the regular fiber class and the class of the exceptional fiber at $z_i$ respectively, then 
$$
c_1(\det \xi)=(2d- 2g+2) F+\sum_i (2\beta_i +1-\alpha_i) F_i.
$$

We shall illustrate the above formula for Seifert $3$-manifolds with the following three examples, which occupy the rest of this section.

\begin{exm}
Let $Y$ be the Seifert $3$-manifold, where $\Sigma$ is the $2$-orbifold with $g=5$,
$\alpha_1=3$, $\alpha_2=5$, $\alpha_3=7$, and $E$ has Seifert invariant $(1,(3,2),(5,3),(7,5))$.
Note that $H^2(Y;\Z)$ has no $2$-torsions, so that each $Spin^c$-structure is uniquely 
determined by its determinant line bundle. With this understood, we shall compute the Seiberg-Witten invariant for the $Spin^c$-structure $\xi$ where
$$
c_1(\det \xi)=-4 F -2F_2-4F_3.
$$
It is easily seen that $\xi$ corresponds to the equivalence class $[D]$ where $D$ is given by the Seifert invariant $(2,(3,1),(5,1),(7,1))$. 

In order to determine the $D^\prime$'s which belong to $[D]$, we first examine the Seifert
invariants of $kE$, $k\in\Z$. Denote the Seifert invariant of $kE$ by $(f_k, (3,a_k), (5,b_k),
(7,c_k))$. It is easy to check that if $k<0$, then $f_k\leq -4$, and if $k>2$, then $f_k>5$, and
for $k=2$, the Seifert invariant of $kE$ is $(5,(3,1), (5,1), (7,3))$. From this analysis, it follows
that the only elements of $[D]$ which contribute nontrivially to the Seiberg-Witten invariant 
are $D_0=D=(2,(3,1),(5,1),(7,1))$, $D_1=(4,(3,0),(5,4),(7,6))$ and $D_2=(7,(3,2),(5,2),(7,4))$.
Consequently,
$$
SW_Y(\xi)=SW_Y([D])=\sum_{i=0}^2 SW_{\s^1\times \Sigma}(D_i)=\sum_{d=2,4,7} (-1)^d
\left (\begin{array}{c}
8\\
d
\end{array}
\right )
=90.
$$
\end{exm}

\begin{exm}
In this example we consider $Y=\s^1\times \s^2$, which is viewed as a Seifert $3$-manifold
over $\Sigma$, where $\Sigma$ is the $2$-orbifold whose underlying space is $\s^2$, with
two singular points of the same multiplicity $\alpha$. The corresponding orbifold complex line 
bundle $E$ has Seifert invariant $(-1, (\alpha,\beta), (\alpha,\alpha-\beta))$, where 
$0<\beta<\alpha$, $\text{gcd }(\alpha,\beta)=1$.

We shall compute the Seiberg-Witten invariant of $Y$ as a Seifert $3$-manifold over $\Sigma$,
where for simplicity we take $\alpha=3$ and $\beta=1$. For each $d$, there are $3$ distinct equivalence classes
$[D_{d,i}]$, $i=0,1,2$, where
$D_{d,0}=(d, (3, 0), (3,0))$, $D_{d,1}=(d, (3,0), (3,1))$ and $D_{d,2}=(d,(3,1), (3,1))$. The 
$Spin^c$-structures $\xi_{d,i}$ which correspond to $[D_{d,i}]$ satisfy 
$c_1(\det \xi_{d,i})=(6d+2+2i)t$, where $t=PD(\s^1\times \{pt\})$. 
In order to compute the Seiberg-Witten invariants, we determine
the elements of each $[D_{d,i}]$. It is straightforward to check that the elements of $[D_{d,0}]$
have Seifert invaraints
$$
(d, (3, 0), (3,0)), \;\; (d-1, (3, 1), (3,2)), \;\; (d-1, (3, 2), (3,1)),
$$ 
the elements of $[D_{d,1}]$ have Seifert invaraints
$$
(d, (3, 0), (3,1)), \;\; (d, (3, 1), (3,0)), \;\; (d-1, (3, 2), (3,2)),
$$
and  the elements of $[D_{d,2}]$ have Seifert invaraints
$$
(d, (3, 1), (3,1)), \;\; (d, (3, 2), (3,0)), \;\; (d, (3, 0), (3,2)).
$$
Consequently, for $d\geq 0$, 
$$
SW_Y(\xi_{d,i})=SW_Y([D_{d,i}])= \left \{\begin{array}{cc}
(d+1)+d+d=3d+1 & \mbox{ if } i=0,\\
(d+1)+(d+1)+d=3d+2 & \mbox{ if } i=1,\\
(d+1)+(d+1)+(d+1)= 3d+3 & \mbox{ if }  i=2,\\
\end{array}
\right .
$$
and for $d<0$, $SW_Y(\xi_{d,i})=0$ for $i=0,1,2$.

The Seiberg-Witten invariant of $Y=\s^1\times\s^2$ is given by the polynomial 
$$
SW_Y(t)= (t^{-1}-t)^{-2}=\sum_{m=1}^\infty m t^{2m}, \mbox{ where } t=PD(\s^1\times \{pt\}).
$$
It is clear that our computation gives the same result. 

\end{exm}

\begin{exm}
In this example $Y$ is the mapping torus of a periodic diffeomorphism $f: T^2\rightarrow T^2$
of order $6$, where $f$ is given by the matrix
$$
\left (\begin{array}{cc}
1 & 1\\
-1 & 0\\
\end{array}
\right ).
$$
Topologically, $Y$ is the $0$-surgery on the trefoil knot. This description allows us to compute
the Seiberg-Witten invariant in terms of the Alexander polynomial of the knot (cf. \cite{MT,FS}),
and it is given by the following polynomial
$$
SW_Y(t)=\frac{t^4-t^2+1}{(1-t^2)^2}=1+\sum_{m=1}^{\infty}m t^{2m},
$$
where $t$ is the Poincar\'{e} dual of the generator of $H_1(Y)$.

On the other hand, since $f$ is periodic, $Y$ can be also viewed as a Seifert $3$-manifold
over $\Sigma$, where $\Sigma$ is the $2$-orbifold whose underlying space is $\s^2$ and which
has $3$ singular points $z_1, z_2, z_3$ with associated multiplicities $2,3$ and $6$. The 
corresponding orbifold complex line bundle $E$ has Seifert invariant $(-2, (2,1), (3,2), (6,5))$.
We shall compute the Seiberg-Witten invariant of $Y$ from this point of view.

First, note that $E$ is torsion of order $6$, with the Seifert invariant of $kE$ 
given by
$$
\begin{array}{cc}
(-2, (2,1), (3,2), (6,5)), & \mbox{ if } k=1,\\
(-1, (2,0), (3,1), (6,4)), & \mbox{ if } k=2,\\
(-1, (2,1), (3,0), (6,3)), & \mbox{ if } k=3,\\
(-1, (2,0), (3,2), (6,2)), & \mbox{ if } k=4,\\
(-1, (2,1), (3,1), (6,1)), & \mbox{ if } k=5.\\
\end{array}
$$
With this understood, for each $d$, there are $6$ distinct equivalence classes $[D_{d,i}]$,
$i=0,1,2,3,4,5$, where in terms of the Seifert invariant,
$$
D_{d,i}=\left \{\begin{array}{cc}
(d, (2,0), (3,0), (6,0)), & \mbox{ if } i=0,\\
(d, (2,0), (3,0), (6,1)), & \mbox{ if } i=1,\\
(d, (2,0), (3,1), (6,0)), & \mbox{ if } i=2,\\
(d, (2,1), (3,0), (6,0)), & \mbox{ if } i=3,\\
(d, (2,1), (3,0), (6,1)), & \mbox{ if } i=4,\\
(d, (2,1), (3,1), (6,0)), & \mbox{ if } i=5.\\
\end{array}
\right .
$$
Let $\xi_{d,i}$ be the $Spin^c$-structure corresponding to $[D_{d,i}]$. Then one can easily 
check that $c_1(\det \xi_{d,i})=(12d +2i)t$. It is straightforward, though tedious, to write down
the Seifert invariants of the elements of each $[D_{d,i}]$, with which one obtains 
$$
SW_Y(\xi_{d,i})=SW_Y([D_{d,i}])=\left \{\begin{array}{cc}
0 & \mbox{ if } d<0,\\
1  & \mbox{ if } d=0, i=0,\\
6d +i & \mbox{ otherwise. }\\
\end{array}
\right.
$$
It is clear that this is the same as given by the polynomial $SW_Y(t)=1+\sum_{m=1}^\infty mt^{2m}$.

\end{exm}

\section{Application: non-K\"{a}hler elliptic surfaces}

We first extend the formula for the Seiberg-Witten invariants of Seifert $3$-manifolds from the
previous section to Seifert fibered $3$-orbifolds. The proof of Theorem 1.4 follows easily from
this with the help of Baldridge's theorem in \cite{Bald2}. 

Let $\pi: Y\rightarrow \Sigma$ be a Seifert fibered $3$-orbifold where $\Sigma$ has genus $g$ 
and $m$ singular points $z_i$ with multiplicities $\alpha_i$, $i=1,\cdots,m$, and where the corresponding orbifold complex line bundle $E$ has Seifert invariant $(e, (\alpha_i,e_i))$. 
Here for each $i$,  $0\leq e_i<\alpha_i$, and $e_i$ and $\alpha_i$ are not necessarily relatively prime. In fact, the components of the singular set of $Y$ are given by exceptional fibers: the fiber 
over $z_i$ lies in the singular set if and only if $\text{gcd }(\alpha_i,e_i)>1$, where in this case, 
$a_i\equiv \text{gcd }(\alpha_i,e_i)$ is the associated multiplicity of the corresponding singular component. 

As we shall see from the proof, the set of $Spin^c$-structures of $Y$ which have nonzero Seiberg-Witten invariant may be identified with a subset of equivalence classes $[D]$, where 
$[D]=[D^\prime]$ if and only if $D-D^\prime\equiv 0 \pmod{E}$, such that under the identification, 
if $\xi$ corresponds to $[D]$ and $D$ has Seifert invariant $(d, (\alpha_i,\beta_i))$, one has 
$$
\det \xi= (2d-2g+2) F+\sum_i (2\beta_i +1-\alpha_i) F_i.
$$
Here $F$ stands for the complex line bundle whose first Chern class is Poincar\'{e} dual to
a regular fiber, and for each $i$, if $a_i\equiv \text{gcd }(\alpha_i,e_i)=1$, $F_i$ stands
for the complex line bundle whose first Chern class is Poincar\'{e} dual to the exceptional 
fiber at $z_i$ (which is not a singular component), and if $a_i\equiv \text{gcd }(\alpha_i,e_i)>1$,  
$F_i$ stands for the orbifold complex line bundle defined as follows: in a neighborhood of the 
exceptional fiber at $z_i$ (which is a singular component with multiplicity $a_i$), $F_i$ is given by the $\Z_{a_i}$-equivariant complex line bundle 
$D^2\times \s^1\times \C$ with the standard $\Z_{a_i}$-action
$$
\lambda\cdot (z,x,w)=(\lambda z,x, \lambda w), \; \lambda=\exp (2\pi i/a_i), \;
(z,x)\in D^2\times \s^1, w\in\C,
$$
and over the rest of $Y$, $F_i$ is defined by extending the trivialization given by the
$\Z_{a_i}$-equivariant nonzero section $(z,x)\mapsto z$ along $\partial D^2\times \s^1$.

\begin{thm}
The Seiberg-Witten invariant of $Y$ is given by
$$
SW_Y(\xi)\equiv SW_Y([D])=\sum_{D^\prime\in [D]} SW_{\s^1\times \Sigma} (D^\prime).
$$
\end{thm}

\begin{proof}
Without loss of generality, we assume that $a_m>1$, which means that the exceptional fiber at 
$z_m$ is a singular component. We let $Y_0$ be the $3$-orbifold obtained by removing the
exceptional fiber at $z_m$ from the singular set of $Y$. Then it follows easily that $Y_0$ is 
Seifert fibered over $\Sigma_0$, where $\Sigma_0$ is the $2$-orbifold obtained from $\Sigma$
by changing the multiplicity of $z_m$ to $\alpha_m/a_m$. The corresponding orbifold complex
line bundle $E_0$ for $Y_0$ has Seifert invariant 
$$
(e, (\alpha_1,e_1), \cdots, (\alpha_{m-1}, e_{m-1}), (\alpha_m/a_m, e_m/a_m)).
$$
In fact, let $\pi:Y\rightarrow \Sigma$ and $\pi_0:Y_0\rightarrow \Sigma_0$ denote the Seifert
fibrations. Then the local models for $\pi$ and $\pi_0$ near $z_m$ are given as follows. 
Consider the $\Z_{\alpha_m}$-action on $D^2\times \s^1$ given by
$$
\lambda\cdot (z,w)=(\lambda z,\lambda^{e_m} w), \; \lambda=\exp (2\pi i/\alpha_m),
(z,w)\in D^2\times \s^1. 
$$
The Seifert fibration $\pi$ is given by the $\Z_{\alpha_m}$-equivariant map $(z,w)\mapsto z$.
Likewise, $\pi_0$ is given by the $\Z_{\alpha_m/a_m}$-equivariant map $(z,w)\mapsto z$,
where the $\Z_{\alpha_m/a_m}$-action on $D^2\times \s^1$ is given by 
$$
\lambda_0\cdot (z,w)=(\lambda_0 z,\lambda_0^{e_m/a_m} w), \; \lambda_0=\lambda^{a_m}
=\exp (2\pi a_m i/\alpha_m), (z,w)\in D^2\times \s^1. 
$$
Now let $\Sigma_{-}$, $\Sigma_{0,-}$ be regular neighborhoods of $z_m$ in $\Sigma$,
$\Sigma_0$ respectively. Then there are decompositions $\Sigma=\Sigma_{+}\cup_\gamma
\Sigma_{-}$ and $\Sigma_0=\Sigma_{+}\cup_\gamma \Sigma_{0,-}$, where $\gamma$ is the
link around $z_m$. These decompositions are compatible with the decompositions 
$Y=Y_{+}\cup_T Y_{-}$ and $Y_0=Y_{+}\cup_T Y_{0,-}$, where $T$ is the $2$-torus over
$\gamma$. Moreover, $\pi=\pi_0: Y_{+}\rightarrow \Sigma_{+}$.

With the preceding understood, we shall reduce the proof of the theorem for $Y$ to that for 
$Y_0$ using Theorem 1.1, and inductively, the theorem is reduced to the case of 
Seifert $3$-manifolds, which we have already established.

To this end, we associate to each orbifold complex line bundle $D$ over $\Sigma$ an orbifold
complex line bundle $D_0$ over $\Sigma_0$ as follows. Suppose $D$ has Seifert invariant 
$(d, (\alpha_i,\beta_i))$. Then we define $D_0$ to be the one which has Seifert invariant
$$
(d, (\alpha_1,\beta_1), \cdots, (\alpha_{m-1}, \beta_{m-1}), (\alpha_m/a_m, [\beta_m/a_m])),
$$
where $\beta_m=[\beta_m/a_m]\cdot a_m +b_m$, with $0\leq b_m<a_m$. We claim that
under the map $\phi$ in Theorem 1.1, the $Spin^c$-structure corresponding to $[D]$ is sent to the 
$Spin^c$-structure corresponding to $[D_0]$.

To see this, let $\xi$, $\xi^0$ be the $Spin^c$-structures corresponding to $[D]$ and $[D_0]$
respectively. Then the associated spinor bundles are $S=\pi^\ast (D\oplus K_\Sigma^{-1}\otimes D)$
and $S^0=\pi_0^\ast (D_0\oplus K_{\Sigma_0}^{-1}\otimes D_0)$. Over $Y_{+}$, $S$ and $S^0$ are
identified as follows. Consider $D$ over $\Sigma_{-}$ and $D_0$ over $\Sigma_{0,-}$. There are
canonical trivializations of $D$ and $D_0$ along $\gamma$, such that with these trivializations
the relative first Chern class of $D$ and $D_0$ over $\Sigma_{-}$ and $\Sigma_{0,-}$ are
$\beta_m/\alpha_m$ and $[\beta_m/a_m]/(\alpha_m/a_m)$ respectively. In fact, these 
trivializations are given by the equivariant nonzero sections $z\mapsto z^{\beta_m}$
and $z\mapsto z^{[\beta_m/a_m]}$ respectively. Now by the definition of $D_0$, $D$ is isomorphic
to $D_0$ over $\Sigma_{+}$ with these trivializations on the boundary 
$\gamma=\partial \Sigma_{+}$. Likewise, $K_\Sigma^{-1}$ and $K_{\Sigma_0}^{-1}$ are 
isomorphic over $\Sigma_{+}$ with certain canonical trivializations on $\gamma$, 
which are given by the
equivariant nonzero section $z\mapsto z$. Note that with respect to these trivializations, the
relative first Chern class of $K_\Sigma^{-1}$ and $K_{\Sigma_0}^{-1}$ over $\Sigma_{-}$ and $\Sigma_{0,-}$ are $1/\alpha_m$ and $1/(\alpha_m/a_m)$ respectively. Now with
$\pi=\pi_0: Y_{+}\rightarrow \Sigma_{+}$, these isomorphisms between $D$, $D_0$ and 
$K_\Sigma^{-1}$, $K_{\Sigma_0}^{-1}$ over $\Sigma_{+}$ give an identification between $S$
and $S^0$ over $Y_{+}$, as a bundle with a fixed trivialization on $T=\partial Y_{+}$. 

With this understood, the claim $\xi^0=\phi(\xi)$ will follow by examining the relative first Chern
classes of $\det S$ and $\det S^0$ over $Y_{-}$ and $Y_{0,-}$ with respect to these trivializations
on $T$. The relative first Chern classes can be calculated using the local models for $\pi$ and
$\pi_0$ near $z_m$ and the equivariant nonzero sections defining these trivializations; their 
evaluations on the $2$-disc $D^2\times \{pt\}$ are (cf. \cite{C0}, Lemma 3.6)
$$
\frac{2\beta_m+1}{a_m}=2[\beta_m/a_m] +\frac{2b_m+1}{a_m} \mbox{ for } \det S,
$$
and 
$$
2[\beta_m/a_m] +1 \mbox{ for } \det S_0.
$$
By Lemma 2.3, the trivializations of $S|_{Y_{-}}$ and $S^0|_{Y_{0,-}}$ on $T$ must be of the form
$h_{-}: \xi|_T\rightarrow \xi_0$ and $h_{0,-}:\xi^0|_T\rightarrow \xi_0$, and moreover, by
Definition 2.4, $\xi^0=\phi(\xi)$ as claimed.

By Theorem 1.1, we have $SW_Y([D])=SW_{Y_0}([D_0])$. Thus the proof of the theorem 
boils down to verifying
$$
\sum_{D^\prime\in [D]} SW_{\s^1\times \Sigma} (D^\prime)
=\sum_{D^\prime_0\in [D_0]} SW_{\s^1\times \Sigma_0} (D^\prime_0).
$$
To see this, first note that under the correspondence $D\mapsto D_0$, $D^\prime\in
[D]$ if and only if $D_0^\prime\in [D_0]$ (where $D^\prime\mapsto D_0^\prime$), 
because $D^\prime-D\equiv 0\pmod{E}$ if and only if $D^\prime_0-D_0\equiv 0\pmod{E_0}$.
Secondly, $SW_{\s^1\times \Sigma} (D^\prime)=SW_{\s^1\times \Sigma_0} (D^\prime_0)$, 
because if $D^\prime$ is given by the Seifert invariant $(d^\prime, (\alpha_i,\beta_i^\prime))$,
then $D_0^\prime$ is given by 
$$
(d^\prime, (\alpha_1,\beta_1^\prime), \cdots, (\alpha_{m-1}, \beta_{m-1}^\prime), 
(\alpha_m/a_m, [\beta_m^\prime/a_m])),
$$
so that both $SW_{\s^1\times \Sigma} (D^\prime)$ and $SW_{\s^1\times \Sigma_0} (D^\prime_0)$
are equal to 
$$
\left \{\begin{array}{ll}
(-1)^{d^\prime} \left (\begin{array}{c}
2g-2\\
d^\prime
\end{array}
\right ) & \mbox{ if } g\geq 1, d^\prime\in [0,2g-2],\\
d^\prime+1 & \mbox{if } g=0, \; d^\prime\geq 0,\\
0 & \mbox{ otherwise. }
\end{array}
\right .
$$

Finally, we remark that the above proof also shows that the equation 
$$
\det \xi= (2d-2g+2) F+\sum_i (2\beta_i +1-\alpha_i) F_i
$$
follows,  with the help of Proposition 2.6, from the corresponding equation for
Seifert $3$-manifolds. This finishes the proof of Theorem 6.1. 

\end{proof}

In the remaining part of this section, we give a proof of Theorem 1.4. To this end, 
let $X$ be a minimal elliptic surface obtained from $E\times C$,
where $E=\C/\Lambda$ and $C$ is a curve of genus $g$, by doing logarithmic transforms
on lifts $x_i\in\C$ of $m_i$-torsion points $\xi_i$ modulo $\Lambda$, $i=1,2,\cdots,n$.
We remark that $X$ is non-K\"{a}hler if and only if $\sum_i x_i\neq 0$. In this case,
$b_2^{+}(X)=2g$. If $X$ is K\"{a}hler, i.e., $\sum_i x_i=0$, then $b_2^{+}(X)=2g+1$.
See \cite{FM} for details. 

For our purpose here it is more convenient to describe $X$ topologically as a principal 
$T^2$-bundle over $\Sigma$, where $\Sigma$ is the $2$-orbifold canonically obtained
from $C$ as follows. Let $t_i$, $i=1,\cdots,n$, be the point of $C$ over which the logarithmic
transform on $x_i$ is performed. Then the singular points of $\Sigma$ consist of $t_i$ with 
multiplicity $m_i$ for all $i$ where $m_i>1$. 

As we explained in the introduction, after fixing a basis ${\bf e}_1$, ${\bf e}_2$ of $\Lambda$,
$X$ gives rise to a pair of orbifold complex line bundles (or equivalently, principal $\s^1$-bundles)
$E_1$, $E_2$ over $\Sigma$. We shall give a more intrinsic description of $E_1$, $E_2$
here. Note that ${\bf e}_1$, ${\bf e}_2$ determine an identification of $T^2=\C/\Lambda$ with
$\s^1\times \s^1=\R/\Z {\bf e}_1 \times \R/\Z {\bf e}_2$. Let $p_1,p_2:T^2\rightarrow \s^1$ be
the corresponding projections onto the first and the second factor respectively, and let $Y_1$,
$Y_2$ be the principal $\s^1$-bundles over $\Sigma$ induced by $p_1, p_2$. Then $E_1$, $E_2$
are the orbifold complex line bundles associated to $Y_1$, $Y_2$. Note that this process can
be reversed: given a pair of principal $\s^1$-bundles $Y_1$, $Y_2$ over $\Sigma$, we obtain
$X$ back as the pull-back bundle of $Y_1\times Y_2\rightarrow \Sigma\times \Sigma$
via the diagonal map $\Sigma\rightarrow \Sigma\times \Sigma$.

If we choose a different basis ${\bf e}^\prime_1$, ${\bf e}^\prime_2$, with $E_1^\prime$,
$E_2^\prime$ being the corresponding orbifold complex line bundles, where
$$
({\bf e}_1, {\bf e}_2)=({\bf e}^\prime_1, {\bf e}^\prime_2) \left (\begin{array}{cc}
a & b\\
c & d\\
\end{array}
\right ), \mbox{ for some } \left (\begin{array}{cc}
a & b\\
c & d\\
\end{array}
\right )\in SL(2;\Z),
$$
then $E_1^\prime= aE_1+bE_2$, $E_2^\prime=cE_1+dE_2$. In particular, the subgroup of
orbifold complex line bundles generated by $E_1,E_2$ coincides with the subgroup generated 
by $E_1^\prime,E_2^\prime$, which was denoted by $\Gamma_X$ in the introduction.

Finally, $X$ may be regarded as a principal $\s^1$-bundle over the $3$-orbifold $Y_2$,
with a bundle morphism $X\rightarrow Y_1$ which induces the Seifert fibration $\pi_2:
Y_2\rightarrow \Sigma$. In other words, the orbifold complex line bundle corresponding to
the Seifert fibration $\pi: X\rightarrow Y_2$ is $\pi_2^\ast E_1$. 

With the preceding understood, the following theorem is a slightly more general version
of Theorem 1.4, where $X$ is allowed to be K\"{a}hler.

\begin{thm}
Suppose $b_2^{+}\geq 1$ (which means that $X$ must be K\"{a}hler when $g=0$). Then the
subset of $Spin^c$-structures which may have nonzero Seiberg-Witten invariant can be 
identified with the set $\{(D)\}$, where $(D)$ is the orbit of an orbifold complex line bundle $D$
under the action of $\Gamma_X$. With this understood,
$$
SW_X((D))=\left \{\begin{array}{ll}
\sum_{D^\prime\in(D), |D^\prime|\in [0,2g-2]}(-1)^{|D^\prime|} \left (\begin{array}{c}
2g-2\\
|D^\prime|
\end{array}
\right ) & \mbox{ if } g\geq 1\\
\sum_{D^\prime\in(D), |D^\prime|\geq 0} (|D^\prime|+1) & \mbox{if } g=0.\\
\end{array}
\right .
$$
Moreover, if $\L$ is the $Spin^c$-structure corresponding to $(D)$, then
$$
c_1(\det \L)=(2d-2g+2)F +\sum_i (2s_i+1-m_i)F_i,
$$
where $F$ stands for a regular fiber and $F_i$ for the fiber at $t_i$ of the elliptic fibration
on $X$, and $(d,(m_i,s_i))$ is the Seifert invariant of $D$.
\end{thm}

We remark that when $b_2^{+}=1$, where this happens if and only if $g=0$ and $X$ is K\"{a}hler, 
the Seiberg-Wtten invariant of $X$ is defined using the Taubes chamber, i.e., using the 
K\"{a}hler form to orient $H^{2,+}(X;\R)$.

\begin{proof}
Applying Baldridge's theorem in \cite{Bald2} to the Seifert fibration $\pi: X\rightarrow Y_2$, and
with Theorem 6.1, we see that the relevant $Spin^c$-structures of $X$ are parametrized by the 
set of equivalence classes $[[D]]$, where $[[D]]=[[D^\prime]]$ if and only if $[D]$ and $[D^\prime]$
differ by a multiple of $[E_1]$. On the other hand, recall from Theorem 6.1 that $[D]$ stands
for the equivalence class of $D$ where $D,D^\prime$ are considered equivalent if and only if
$D-D^\prime\equiv 0\pmod{E_2}$, or equivalently, $\pi_2^\ast D=\pi_2^\ast D^\prime$.
It follows easily that the equivalence class $[[D]]$ coincides with the orbit $(D)$. Furthermore, 
by Baldridge's formula in \cite{Bald2} as well as the formula in Theorem 6.1, 
\begin{eqnarray*}
SW_X((D)) & = & \sum_{[D_1]\in [[D]]} SW_{Y_2} ([D_1])=\sum_{[D_1]\in [[D]]}
\sum_{D_2\in [D_1]} SW_{\s^1\times\Sigma} (D_2)\\
& = & \sum_{D^\prime\in (D)} 
SW_{\s^1\times\Sigma}(D^\prime),\\
\end{eqnarray*}
from which the claimed formula for $SW_X((D))$ in Theorem 6.2 follows. Finally, the formula for
$c_1(\det \L)$ follows directly from the corresponding formula for Seifert fibered $3$-orbifolds 
(cf. Theorem 6.1).

\end{proof}

\vspace{3mm}

{\bf Acknowledgements:} 
I wish to thank Cliff Taubes for helpful communications.
The work was finished during the sabbatical leave at the Institute for Advanced Study in Fall
2011. I wish to thank the institute for its hospitality and financial support (provided by The S.S. 
Chern Fund and through NSF grant DMS-0635607).


\begin{thebibliography}{99}
\bibitem{Bald1} S. Baldridge, {\em Seiberg-Witten invariants of $4$-manifolds
                with free circle actions}, Comm. Contemp. Math. {\bf 3},
                No. 3 (2001), 341-353.

\bibitem{Bald2} ----------------, {\em Seiberg-Witten invariants, orbifolds, and
                 circle actions}, Tran. of AMS {\bf 355}, No. 4 (2002), 
               1669-1697.

\bibitem{BPV} W. Barth, K. Hulek, C. Peters and A. Van de Ven, 
                     {\em Compact Complex Surfaces}, 2nd enlarged edition,
                       Ergebnisse der Math. Vol. 4, Springer-Verlag, 2004.

\bibitem{Biq} O. Biquard, {\em Les \'{e}quations de Seiberg-Witten sur une surface complexe 
non K\"{a}hl\'{e}rienne}, Comm. Anal. Geom. {\bf 6} (1998), 173-197.

\bibitem{BLP} M. Boileau, B. Leeb and J. Porti, {\em Geometrization 
of 3-dimensional orbifolds}, Ann. of Math. (2) {\bf 162}(2005),  no. 1, 195--290.

\bibitem{BMP} M. Boileau, S. Maillot and J. Porti, {\em Three-dimensional 
orbifolds and their geometric structures}, Panoramas et Synthèses {\bf 15}, 
Société Mathématique de France, Paris, 2003.

\bibitem{B} R. Brussee, {\em The canonical class and the $C^\infty$ properties of K\"{a}hler
 surfaces}, New York J. Math. {\bf 2} (1996), 103-146. 
               
\bibitem{C0} W. Chen, {\em Orbifold adjunction formula and symplectic cobordisms 
                  between lens spaces}, Geometry and Topology  {\bf 8} 
                  (2004), 701-734.

\bibitem{D} S. Donaldson, {\em The Seiberg-Witten equations and $4$-manifold topology}, 
Bull. Amer. Math. Soc. (N.S.) {\bf 33} (1996), no.1, 45-70. 


\bibitem{FS} R. Fintushel and R. Stern, {\em Knots, links, and
             $4$-manifolds}, Invent. Math.{\bf 134} (1998), 363-400.
             
\bibitem{FM} R. Friedman and J.W. Morgan, {\em Smooth Four-manifolds and Complex Surfaces},
 Ergebnisse der Math., Springer-Verlag, 1994.
                     
\bibitem{FM1} -----------------, {\em Obstruction bundles, semi regularity, and Seiberg-Witten invariants}, Comm. Ana. Geom. {\bf 7} (1999), no.3, 451-495. 
                    
\bibitem{LL} T.J. Li and A. Liu, {\em General wall crossing formula}, Math. Res. Lett. {\bf 2} (1995),
               797-810.

\bibitem{MM} D. McCullough and A. Miller, {\em Manifold covers of $3$-orbifolds with 
                      geometric pieces},
                     Topology and its Applications {\bf 31} (1989), 169-185.                                  


\bibitem{MT} G. Meng and C.H. Taubes, {\em \underline{SW}= Milnor torsion},
             Math. Res. Lett. {\bf 3} (1996), 661-674.

\bibitem{M} J.W. Morgan, {\em The Seiberg-Witten Equations and Applications 
               to the Topology of Smooth Four-Manifolds}, Mathematical Notes
               {\bf 44}, Princeton Univ. Press, Princeton, New Jersey, 1996.

\bibitem{MMS} J.W. Morgan, T.S. Mrowka and Z. Szabo, {\em Product formulas 
                  along $T^3$ for Seiberg-Witten invariants}, 
                    Math. Res. Lett. {\bf 4} (1997), 915-929.  
                    
\bibitem{MOY} T. Mrowka, P. Ozsv\'{a}th and B. Yu, {\em Seiberg-Witten monopoles on
            Seifert fibered spaces}, Comm. Ana. Geom. {\bf 5} (1997), no.4, 685-793.     

\bibitem{Pe} C. Petronio, {\em Spherical splitting of $3$-orbifolds}, Math. Proc. Camb. Phil. 
Soc. (2007) {\bf 142}, 269-287.                            
                    
\bibitem{T1} C.H. Taubes, {\em $SW\Rightarrow Gr$: from the Seiberg-Witten
                          equations to pseudoholomorphic curves}, J. Amer.
                           Math. Soc. {\bf 9} (1996), 845-918.         

\bibitem{T2} -----------, {\em The Seiberg-Witten invariants and $4$-manifolds 
                           with essential tori}, Geometry and Topology 
                           {\bf 5} (2001), 441-519.
                           
\bibitem{Tu1} V.  Turaev, {\em Torsion invariants of $Spin^c$-structures on $3$-manifolds}, 
Math. Res. Lett. {\bf 4} (1997), 679-695.    

\bibitem{Tu2} -------------, {\em A combinatorial formulation for the Seiberg-Witten invariants of 
$3$-manifolds}, Math. Res. Lett. {\bf 5} (1998), 583-598.                         

\bibitem{W} S.R. Williams, {\em The Seiberg-Witten Invariant on Non-K\"{a}hler Complex 
Surfaces}, Thesis (Ph.D.)--University of Adelaide, Dept. of Pure Mathematics, 1997.
           
\end{thebibliography}
\end{document}